\documentclass[10pt]{amsart}
\usepackage{setspace}
\usepackage[utf8]{inputenc}
\usepackage{graphicx}
\graphicspath{ {./images/} }
\usepackage[pagewise]{lineno}

\usepackage{amssymb}
\usepackage[foot]{amsaddr}
\usepackage[all,arc]{xy}
\usepackage{enumerate}
\usepackage{mathrsfs}
\usepackage[margin=1in]{geometry} 
\usepackage{amsmath,amsthm}
\usepackage{mathtools}
\usepackage[T1]{fontenc}
\usepackage{fourier}
\usepackage{stmaryrd}
\usepackage{amsfonts}

\usepackage[backref=page]{hyperref}
\usepackage[capitalise,nameinlink]{cleveref}
\newcommand*\bb{\mathbb}

\newcommand *\w{^\wedge}
\newcommand*\g{\mathfrak{g}}

\DeclareMathOperator*{\esup}{ess\,sup}

\DeclareMathOperator*{\eosc}{ess\,osc}

\newcommand*\de{\partial}

\newcommand*\tl{\mathrm{Tail}}

\newtheorem{thm}{Theorem}[section]
\theoremstyle{definition}
\newtheorem{defn}[thm]{Definition}

\newtheorem{cor}[thm]{Corollary}
\newtheorem{prop}[thm]{Proposition}
\newtheorem{lem}[thm]{Lemma}
\newtheorem{rem}[thm]{Remark}

\usepackage{xcolor}
\hypersetup{
	colorlinks,
	linkcolor={red!50!black},
	citecolor={red!50!black},
	urlcolor={blue!80!black}
}

\DeclarePairedDelimiterX{\inp}[2]{\langle}{\rangle}{#1, #2}

\makeatletter
\def\@tocline#1#2#3#4#5#6#7{\relax
	\ifnum #1>\c@tocdepth 
	\else
	\par \addpenalty\@secpenalty\addvspace{#2}%
	\begingroup \hyphenpenalty\@M
	\@ifempty{#4}{%
		\@tempdima\csname r@tocindent\number#1\endcsname\relax
	}{%
		\@tempdima#4\relax
	}%
	\parindent\z@ \leftskip#3\relax \advance\leftskip\@tempdima\relax
	\rightskip\@pnumwidth plus4em \parfillskip-\@pnumwidth
	#5\leavevmode\hskip-\@tempdima
	\ifcase #1
	\or\or \hskip 1em \or \hskip 2em \else \hskip 3em \fi%
	#6\nobreak\relax
	\dotfill\hbox to\@pnumwidth{\@tocpagenum{#7}}\par
	\nobreak
	\endgroup
	\fi}
\makeatother
\usepackage{hyperref}

\newcommand{\overbar}[1]{\mkern 1.5mu\overline{\mkern-1.5mu#1\mkern-1.5mu}\mkern 1.5mu}

\numberwithin{equation}{section}
\usepackage{epigraph}

\title{Local H\"older regularity for nonlocal porous media and fast diffusion equations with general kernel}

\author{Kyeongbae Kim $^{\ast}$}
\author{Ho-Sik Lee$^{\dagger}$}
\author{Harsh Prasad$^{\S}$}
\address{$^{\ast}$ Department of Mathematical Sciences, Seoul National University, Seoul 08826, Korea.}
\address{$^{\dagger}$$^{\S}$ Fakultät für Mathematik, Universität Bielefeld, Postfach 100131, D-33501 Bielefeld, Germany.}
\email{$^{\ast}$ kkba6611@snu.ac.kr}
\email{$^{\dagger}$ ho-sik.lee@uni-bielefeld.de}
\email{$^{\S}$ hprasad@math.uni-bielefeld.de}
\thanks{$^{\ast}$ Support from the National Research Foundation of Korea (NRF) through IRTG 2235/NRF-2016K2A9A2A13003815 at Seoul National University.}
\thanks{$^{\dagger}$ Support from funding of the Deutsche Forschungsgemeinschaft through GRK 2235/2 2021 - 282638148.}
\thanks{$^{\S}$ Support from Deutsche Forschungsgemeinschaft (DFG, German Research Foundation) – Project-ID 317210226 – SFB 1283 is gratefully acknowledged}
\date{April 2025}
\begin{document}
\begin{abstract}
We show that locally bounded, local weak solutions to certain nonlocal, nonlinear diffusion equations modeled on the fractional porous media and fast diffusion equations given by
\begin{align*}
\de_t u + (-\Delta)^s(|u|^{m-1}u) = 0 \quad \mbox{ for } \quad 0<s<1 \quad\text{and}\quad m>0
\end{align*}
are locally H\"older continuous. We work with bounded, measurable kernels and provide the corresponding $L^{\infty}_{loc} \rightarrow C^{0,\alpha}_{loc}$ De Giorgi-Nash-Moser theory for the equation via a delicate analysis of the set of singularity/degeneracy in a geometry dictated by the solution itself and a careful analysis of far-off effects. In particular, our results are in the spirit of interior regularity, requiring the equation to hold only locally, and thus are new even for positive solutions of the equation with constant coefficients. 
\end{abstract}
\maketitle

\noindent\textbf{\textup{2020} Mathematics Subject Classification.} {Primary: 
35R09, 
35B65; 
Secondary: 
47G20, 
35K55
}

\smallskip
	
\tableofcontents
\section{Introduction}
\subsection{The Problem} We are interested in studying the equation 
\begin{equation}\label{eq:main-eq}
\de_t u - L\phi(u) = 0\quad\text{in }\Omega\times (0,T],
\end{equation}
where $\Omega\subset\mathbb{R}^n$ ($n\in\mathbb{N}$ and $n\geq 2$) is bounded and open, $T>0$, and the nonlinearity $\phi:\mathbb{R}\rightarrow\mathbb{R}$ is of the following type 
\begin{align}\label{eq:phi}
\phi(X) =|X|^{m-1}X \quad \mbox{ for } \quad m>0
\end{align}
and the operator $L$ is modeled on the fractional Laplacian and given by the principal value integral
\begin{align*}
Lv(x,t) = \text{p.v.}\int_{\bb{R}^n}(v(y,t)-v(x,t))K(x,y;t) \,dy
\end{align*}
with $n\geq 2$ for a kernel $K:\bb{R}^n\times\bb{R}^n\times\bb{R} \rightarrow [0,\infty)$ which is measurable and satisfies
\begin{equation}\label{eq:kernel}
\lambda |x-y|^{-n-2s} \leq K(x,y;t) \leq \Lambda |x-y|^{-n-2s}
\end{equation}
for two fixed positive constants $\lambda, \Lambda > 0$. Here and throughout the paper, we always assume that 
\begin{align*}
s\in(0,1).
\end{align*}	
\subsection{The Context} The equation \cref{eq:main-eq} is modelled on the fractional diffusion equation (FDE) given by
\begin{align*}
    \partial_t u + (-\Delta)^s(|u|^{m-1}u) = 0 \quad (s\in(0,1), m>0),
\end{align*}
which is itself a nonlocal analogue of the classical diffusion equation
\begin{align*}
    \partial_t u -\Delta(|u|^{m-1}u) = 0.
\end{align*}
The nonlinear local diffusion equation finds applications in modelling density-dependent diffusive behaviours, such as plasma diffusion across magnetic fields \cite{Oku73, Dra77} and phenomena described by scaling laws and self-similarity as detailed in the theory of the porous medium equation \cite{Bar96, Vas07}. On the other hand, nonlocality, introduced via the fractional Laplacian $(-\Delta)^s$, is crucial to capture phenomena involving long-range interactions or anomalous diffusion, relevant in areas such as mechanics and physics involving constraints \cite{lions72}, stochastic hydrodynamic limits \cite{Gia99}, and broader frameworks of nonlocal diffusion processes \cite{Caf12}.

The nonlocal FDE has been a subject of extensive mathematical investigation. Foundational studies established the well-posedness of the Cauchy problem for the fractional porous medium equation \cite{dePab11, dePab12}, explored existence and uniqueness for measure data \cite{Gri15}, analysed distributional solutions \cite{del17}, and investigated fundamental solutions and asymptotic behaviour like Barenblatt profiles \cite{Vaz14}, including optimal estimates for fast diffusion \cite{Vaz15}. There have also been results demonstrating H\"older continuity for singular parabolic cases \cite{kim11} using extension arguments, deriving quantitative local and global estimates \cite{Bon14}, establishing higher regularity and classical solutions \cite{Vaz17}, studying regularity for equations with rough kernels \cite{dePab16} or singular diffusion \cite{dePab18} for equations posed in the full space, analysing propagation effects like infinite speed \cite{Bon17}, investigating transformations of self-similar solutions \cite{Sta15}, and considering the Cauchy problem in bounded domains \cite{Bon23}. Numerical analysis of the equation has also progressed, with works developing robust numerical schemes based on semigroup theory \cite{Cusi20} and finite difference methods, including theoretical analysis and experimental validation \cite{num19, num18}.

We are interested in studying local regularity along the lines of the De Giorgi-Nash-Moser theory. More specifically, we are interested in local H\"older bounds for locally bounded weak solutions to \cref{eq:main-eq}. In the linear setup when $m=1$, analogous regularity results were previously established: a priori estimates for integro-differential operators provided regularity for the elliptic case \cite{Kas09}, and regularity theory, including H\"older continuity, was developed for the linear parabolic case \cite{CCV11}. For nonlocal parabolic equations with degeneracy and singularity dependent on a fractional '\emph{gradient}' term (like the fractional p-Laplacian), H\"older regularity was obtained in \cite{apt, Lia24}. On the other hand, equation \eqref{eq:main-eq} exhibits singularity for the fast diffusion case ($m<1$) and degeneracy for the porous media case ($m>1$) - both directly at the level of the solution $u$ itself.
\subsection{Main Ideas} The key issues in studying the local regularity of \cref{eq:main-eq} are:
\begin{itemize}
\item[1.] The singularity/degeneracy of the term $|u|^{m-1}$ when $u \approx 0$.
\item[2.] The nonlocal nature of the operator wherein far off values of the solution $u$ can affect its regularity in a small open set.
\item[3.] The absence of a De Giorgi isoperimetric inequality for fractional Sobolov spaces which is crucial in the local setup.   
\end{itemize}
We note that even if the solution is assumed to be bounded, the singularity/degeneracy can still cause problems since the solution can be close to $0$. To deal with this, we separate our analysis into two phases - one where the solution is close to zero and one where it is away - and employ a variant of the switching radius argument. More precisely, we first study the solution close to zero and show that the equation forces a reduction in the supremum as we zoom in; further, the reduction continues as long as the solution stays close to zero. In particular, if the solution is always close to zero, then we are done. Thus we now switch to the case where the solution is away from zero in some cylinder and at this point, it turns out that we are almost in the linear setup from which we can conclude oscillation decay. Thus, the analysis boils down to studying the solution $u$ in the set $\{u\approx 0\}$. We do this via suitably scaling the cylinders to match the singularity/degeneracy in terms of the solution itself i.e. we work in a geometry dictated by the solution. These ideas are not new, and our work draws much from previous works on the subject such as the study of $C^{1+\alpha}$ regularity for (elliptic) $p-$Laplace equation \cite{DiB83} and $C^{\alpha}$ regularity for porous media systems \cite{Lia21}; these ideas themselves, of course, go back to De Giorgi's seminal work \cite{DeG}.  
	
For the non-locality, one could impose global boundedness but we are interested in local regularity, thus it is not an appropriate assumption to make. Furthermore, unlike the linear case, one cannot go from global boundedness to local boundedness with the addition of a source term due to the nonlinearity of $\phi$. 

Finally, the De Giorgi isoperimetric inequality is an important tool in regularity theory for local equations; however, such an isoperimetric inequality in the fractional case necessarily comes with a jump term \cite[Theorem 4]{apt-ellip}. In the absence of the inequality, one must usually resort to logarithmic estimates to control jumps, but such estimates are currently unavailable in the nonlocal nonlinear parabolic setup. To ameliorate the situation, we get a "good term" or the isoperimetric term in the energy estimate itself - the idea first appeared in \cite{CCV11}; we make use of a variant as in \cite{apt}. With the isoperimetric inequality at hand, we are left to deal with the non-locality of the operator itself. We do this via tail alternatives - which say that if the tail is not too large then we get the estimate we want. Finally, we need to make sure that the tails are in fact not too large and we are able to get this by zooming further and further away followed by delicately adding up the contributions. 

\subsection{Main Results} We now state our main results and provide an outline of the paper. 
Denoting
\begin{align*}
\textbf{data}=\{n,s,m,\lambda,\Lambda\},
\end{align*}
and using the notion of weak solution defined in Section \ref{sec:prel}, we now state our main theorem.

\begin{thm}\label{thm:hol}
For $s\in(0,1)$ and $m>0$, let $u$ be a locally bounded, local weak solution to \cref{eq:main-eq} in $\Omega_T$ with \cref{eq:phi} and \cref{eq:kernel}. Then $u\in C^{0,\alpha}_{\text{loc}}(\Omega_T)$ for some $\alpha\in(0,1)$ depending on $\textbf{data}$. 
\end{thm}

\begin{rem}
The corresponding estimate for Theorem \ref{thm:hol} is in Theorem \ref{thm:hol2}.
\end{rem}
\begin{rem}
Our result is independent of any initial or boundary value that may be prescribed and does not depend on the regularity of the coefficients. It is an interior regularity result - the equation is posed only locally - in the spirit of De Giorgi-Nash-Moser theory. 
\end{rem}
\begin{cor}[Liouville]\label{cor:Lio}
If $u$ solves \cref{eq:main-eq} in $\bb{R}^n\times(-\infty,T)$ for some $-\infty<T<\infty$ and if $u$ is bounded then $u$ is a constant. 
\end{cor}

The paper is organized as follows. In Section \ref{sec:prel}, we provide notation, definitions, and auxiliary lemmas. Section \ref{sec:cacc} is devoted to a number of energy estimates\footnote{ We need to work with more than one energy estimate for the same equation. At least in the fast diffusion case, it is unclear what the corresponding De Giorgi class should be.}. Section \ref{sec:De} is used to prove variants of De Giorgi's lemma in our setup. The equation is analysed near its singularity/degeneracy in its intrinsic geometry in Section \ref{sec:close0}. Finally, in Section \ref{sec:hol}, we prove Theorem \ref{thm:hol} with the estimate in Theorem \ref{thm:hol2} along with \cref{cor:Lio}.

\section{Preliminaries}\label{sec:prel}
Let $c$ be a positive generic constant which depends on the \textbf{data}. We will specify its dependence using brackets, e.g., $c=c(n,s,m,\lambda,\Lambda)$ means that $c$ depends on $n,s,m,\lambda$, and $\Lambda$. We will often refer to constants that depend only on \textbf{ data} as universal constants. For positive numbers or functions $f$ and $g$, we write $f\eqsim g$ when there exists an implicit constant $c\geq 1$ such that $\frac{1}{c}f\leq g\leq cf$. For $U\subset\mathbb{R}^n$ an open set and $T>0$, we write $U_T=U\times(0,T]$.

We mention that if $a,b\in\mathbb{R}$ and $m>0$, then there exists $c=c(m)\geq 1$ such that we have
\begin{align*}
\frac{1}{c}(|a|+|b|)^{m-1}(a-b)\leq\phi(a)-\phi(b)\leq c(|a|+|b|)^{m-1}(a-b)
\end{align*}
from the relation $(\phi(a)-\phi(b))(a-b)\eqsim (|a|+|b|)^{m-1}(a-b)^2$ in \cite[Appendix B]{DieForTomwan20}.

Throughout the paper, we always assume $n\in\mathbb{N}$ with $n\geq 2$. For $\rho>0$, $x=(x_1,\dots,x_n)\in\mathbb{R}^n$, and $x_0=(x_{0,1},\dots,x_{0,n})\in\mathbb{R}^n$, we denote
\begin{align*}
K_{\rho}(x_0)=\{x\in\mathbb{R}^n:|x_i-x_{0,i}|<\rho\quad\text{for all }i=1,2,\dots,n\},
\end{align*}
i.e., $K_{\rho}(x_0)$ is the cube centred at $x_0$ and has radius $\rho$.	We write its boundary as $\de K_{\rho}(x_0)$. For $t_0\in\mathbb{R}$ and $s\in(0,1)$, we write the open time interval as follows:
\begin{align*}
I_{\rho}(t_0)=I^{s}_{\rho}(t_0)=(t_0-\rho^{2s},t_0).
\end{align*}
We will simply write the cube as $K$ and the time interval as $I$ if their centres and the radii are clear in the context. 
Now with denoting $z_0=(t_0,x_0)\in\mathbb{R}^{n+1}$, with $\rho,\tau>0$ we define time-spatial cube
\begin{align*}
Q_{\rho,\tau}(z_0)=I_{\tau}(t_0)\times K_{\rho}(x_0)\quad\text{and}\quad Q_{\rho}(z_0)=Q_{\rho,\rho}(z_0).
\end{align*}
We usually write $Q_{\rho,\tau}=Q_{\rho,\tau}(0)$ and $Q_{\rho,\tau}(x_0,t_0)=(x_0,t_0)+Q_{\rho,\tau}$.
The parabolic boundary $\partial_PQ_{\rho,\tau}(z_0)$ is defined as 
\begin{align*}
\partial_PQ_{\rho,\tau}(z_0)=(\{t_0-\tau^{2s}\}\times K_{\rho}(x_0))\cup ((t_0-\tau^{2s},t_0)\times\partial K_{\rho}(x_0)).
\end{align*}

With $U\subset\mathbb{R}^n$ an open set and $T>0$, we define the function space
\begin{align*}
C(0,T; L^2_{\text{loc}}(U))=\left\{f\in L^1(U_T)\,\,:\,\,\text{for any compact }K\Subset U,\,\,t\mapsto\|f(\cdot,t)\|_{L^2(K)}\,\,\text{is continuous on }(0,T)\right\}.
\end{align*}
With $s\in(0,1)$ and $U\subset\mathbb{R}^n$ an open set, the fractional Sobolov space $W^{s,2}(U)$ is defined as 
\begin{align*}
W^{s,2}(U)=\left\{f\in L^2(U)\,\,:\,\,\|f\|_{L^2(U)}+[f]_{W^{s,2}(U)}<\infty\right\}
\end{align*}
equipped with the norm
\begin{align*}
\|f\|_{W^{s,2}(U)}:=\|f\|_{L^2(U)}+[f]_{W^{s,2}(U)}:=\left(\int_{U}|f|^2\,dx\right)^{\frac{1}{2}}+\left(\int_{U}\int_{U}\dfrac{|f(x)-f(y)|^2}{|x-y|^{n+2s}}\,dx\,dy\right)^{\frac{1}{2}}.
\end{align*}
By $W^{s,2}_0(U)$ we mean functions $ f \in W^{s,2}(\bb{R}^n)$ with $f\equiv 0$ in $\mathbb{R}^n\setminus U$.
Now we define
\begin{align*}
L^2_{\text{loc}}(0,T; W^{s,2}_{\text{loc}}(U))=\left\{f\in L^1(U_T)\,\,:\,\,\text{for any compact }I'\Subset (0,T)\text{ and } K\Subset U,\quad \int_{I'}\|f(\cdot,t)\|^2_{W^{s,2}(K)}\,dt<\infty\right\}.
\end{align*}
	
With $m>0$, $s\in(0,1)$, a time interval $I\subset\mathbb{R}$, and the function space
\begin{align*}
L^{m}_{2s}(\mathbb{R}^n)=\left\{f\in L^{1}_{\text{loc}}(\mathbb{R}^n)\,\,:\,\,\int_{\mathbb{R}^n}\dfrac{|f(y)|^m}{(1+|y|)^{n+2s}}\,dy<\infty\right\},
\end{align*}
we consider the tail space as follows:
\begin{align*}
L^{\infty}(I;L^{m}_{2s}(\mathbb{R}^n))=\left\{f\in L^{1}_{\text{loc}}(\mathbb{R}^{n+1})\,\,:\,\,\esup_{t\in I}\int_{\mathbb{R}^n}\dfrac{|f(y,t)|^m}{(1+|y|)^{n+2s}}\,dy<\infty\right\}.
\end{align*}
For $r>0$, $x_0\in\mathbb{R}^n$ and $f\in L^{\infty}(I;L^{m}_{2s}(\mathbb{R}^n))$, we define the tail of $f$ as follows:
\begin{align*}
\tl(f;K_{r}(x_0)\times I) = \left(r^{2s}\esup_{t\in I} \int_{\bb{R}^n\setminus K_r(x_0)}\frac{|f(y,t)|^m}{|y-x_0|^{n+2s}}\,dy\right)^{\frac{1}{m}}.
\end{align*}

Then our notion of the weak solution to \cref{eq:main-eq} is as follows. 
\begin{defn}\label{def:weak}
Let $m>0$, $s\in(0,1)$, $T>0$, and $\Omega\subset\mathbb{R}^n$ a bounded and open set. A real-valued function $u$ with 
\begin{align}\label{eq:space0}
\phi(u)\in L^2_{\text{loc}}(0,T; W^{s,2}_{\text{loc}}(\Omega)),\quad u\in L^{\infty}(0,T;L^{m}_{2s}(\mathbb{R}^n))
\end{align}
and
\begin{align}\label{eq:space}
\begin{cases}
u\in C(0,T; L^2_{\text{loc}}(\Omega))\cap L^2_{\text{loc}}(0,T; W^{s,2}_{\text{loc}}(\Omega))&\quad\text{when}\quad\,m>1,\\
u\in C(0,T; L^{m+1}_{\text{loc}}(\Omega))&\quad\text{when}\quad 0<m<1
\end{cases}
\end{align}
is a local weak sub(super)-solution to \cref{eq:main-eq} if for every compact set $K\subset \Omega$ and every interval $[t_1,t_2]\subset(0,T]$,
\begin{align*}
\int_{K}u\xi\,dx\Big\lvert^{t_2}_{t_1}-\iint_{K\times(t_1,t_2)}u\xi_t\,dx\,dt+\iiint_{K\times \mathbb{R}^n\times (t_1,t_2)}(\phi(u)(x,t)-\phi(u)(y,t))(\xi(x,t)-\xi(y,t))K(x,y;t)\,dx\,dy\,dt\leq(\geq) 0
\end{align*}
for all nonnegative test functions $\xi\in W^{1,2}_{\text{loc}}(0,T; L^2(K))\cap L^2_{\text{loc}}(0,T; W^{s,2}_0(K))$.
\end{defn}
\begin{rem}
The above statement means that if $u$ is a local weak sub-solution then the inequality `$\leq$' holds, whereas if it is a local weak super-solution in $\Omega_T$, then the inequality `$\geq$' holds in \cref{eq:moleq}. We will use a similar writing style throughout the paper for brevity. 
\end{rem}
\begin{rem}
    We will always assume that the solutions $u$ are locally bounded. The constants will not depend on the local bound of $u$ (but see \cref{rem:no-prob}.) In the porous media case, a standard Moser iteration will show that solutions are always locally bounded, whereas in the fast diffusion case one needs further integrability hypothesis on $u$ to show local boundedness. 
\end{rem}
We frequently use the following Sobolev embedding which are easily proved from modifications of \cite[Eq. (2.4)]{DinZhaZho21} in the proof of \cite[Lemma 2.3]{DinZhaZho21}.
\begin{lem}
Let $n\geq 2$. For $(x_0,t_0)\in\mathbb{R}^{n+1}$, $\rho>0$, and $\tau>0$, let us write $K_{\rho}=K_{\rho}(x_0)$ and $I_{\tau}=I_{\tau}(t_0)$. For $f\in L^2(I_{\tau};W^{s,2}_0(K_{\rho}))\cap L^{\infty}(I_{\tau};L^2(K_{\rho}))$ there holds
\begin{align}\label{eq:SP1}
\int_{I_{\tau}}\int_{K_{\rho}} |f|^{2\frac{n+2s}{n}}\,dx\,dt \leq C\left(\int_{I_{\tau}}\int_{K_{\rho}}\int_{K_{\rho}} \frac{|f(x,t)-f(y,t)|^2}{|x-y|^{n+2s}}\,dx\,dy\,dt \right) \left(\esup_{t\in I_{\tau}}\int_{K_{\rho}} |f(x,t)|^2\,dx\right)^{\frac{2s}{n}}.
\end{align}
Also, for $f\in L^2(I_{\tau};W^{s,2}_0(K_{\rho}))\cap L^{\infty}(I_{\tau};L^1(K_{\rho}))$ there holds
\begin{align}\label{eq:SP2}
\int_{I_{\tau}}\int_{K_{\rho}} |f|^{2\frac{n+s}{n}}\,dx\,dt \leq C\left(\int_{I_{\tau}}\int_{K_{\rho}}\int_{K_{\rho}} \frac{|f(x,t)-f(y,t)|^2}{|x-y|^{n+2s}}\,dx\,dy\,dt \right) \left(\esup_{t\in I_{\tau}}\int_{K_{\rho}} |f(x,t)|\,dx\right)^{\frac{2s}{n}}.
\end{align}
The constants $C$ in the above estimates only depend on $n$ and $s$.
\end{lem}
\begin{proof}
We prove the inequalities for $\rho=1$, since when $\rho\neq 1$, they follow from scaling arguments. By H\"{o}lder inequality together with $n\geq 2$,
\begin{align*}
\int_{I_{\tau}}\int_{K_1}|f(x,t)|^{2\frac{n+2s}{n}}\,dx\,dt&=\int_{I_{\tau}}\int_{K_1}|f(x,t)|^{\frac{4s}{n}}|f(x,t)|^{2}\,dx\,dt\\
&\leq \int_{I_{\tau}}\left(\int_{K_1}|f(x,t)|^2\,dx\right)^{\frac{2s}{n}}\left(\int_{K_1}|f(x,t)|^{\frac{2n}{n-2s}}\,dx\right)^{\frac{n-2s}{n}}\,dt\\
&\leq \left(\esup_{t\in I_{\tau}}\int_{B_1}|f(x,t)|^2\,dx\right)^{\frac{2s}{n}}\int_{I_{\tau}}\left(\int_{K_1}|f(x,t)|^{\frac{2n}{n-2s}}\,dx\right)^{\frac{n-2s}{n}}\,dt
\end{align*}
holds. Now by applying fractional Sobolev-Poincar\'{e} inequality \cite[Theorem 6.5]{DiPalVal12}, we obtain \cref{eq:SP1}. For \cref{eq:SP2}, we estimate
\begin{align*}
\int_{I_{\tau}}\int_{K_1}|f(x,t)|^{2\frac{n+s}{n}}\,dx\,dt&=\int_{I_{\tau}}\int_{K_1}|f(x,t)|^{\frac{2s}{n}}|f(x,t)|^{2}\,dx\,dt\\
&\leq \left(\esup_{t\in I_{\tau}}\int_{B_1}|f(x,t)|\,dx\right)^{\frac{2s}{n}}\int_{I_{\tau}}\left(\int_{K_1}|f(x,t)|^{\frac{2n}{n-2s}}\,dx\right)^{\frac{n-2s}{n}}\,dt.
\end{align*}
Then again applying fractional Sobolev-Poincar\'{e} inequality yields \cref{eq:SP2}.
\end{proof}

We also provide the following type of isoperimetric inequality which is a variant of \cite[Lemma 3.3]{apt}
\begin{lem}\label{lem:iso}
Let $k,l,q\in\mathbb{R}$ be such that $k<l<q$. Then with $C=C(n,s)$,
\begin{align*}
(q-l)(\phi(l)-\phi(k))|\{u>q\}\cap K_{\rho}||\{u<k\}\cap K_{\rho}| \leq C\rho^{n+2s}\int_{K_{\rho}}(u-l)_{+}(x)\int_{K_{2\rho}} \frac{(\phi(u)-\phi(l))_{-}(y)}{|x-y|^{n+2s}}\,dy\,dx.
\end{align*}
\end{lem} 
\begin{proof}
We estimate
\begin{align*}
\int_{K_{\rho}}(u-l)_{+}(x)\int_{K_{2\rho}} \frac{(\phi(u)-\phi(l))_{-}(y)}{|x-y|^{n+2s}}\,dy\,dx&\geq \frac{C}{\rho^{n+2s}}\int_{K_{\rho}}(u-l)_{+}(x)\int_{K_{2\rho}}(\phi(u)-\phi(l))_{-}(y)\,dy\,dx\\
&\geq \frac{C}{\rho^{n+2s}}\int_{K_{\rho}\cap\{u>q\}}(u-l)_{+}(x)\int_{K_{\rho}\cap\{u<k\}}(\phi(u)-\phi(l))_{-}(y)\,dy\,dx\\
&\geq \frac{C}{\rho^{n+2s}}\int_{K_{\rho}\cap\{u>q\}}(q-l)\int_{K_{\rho}\cap\{u<k\}}(\phi(l)-\phi(k))\,dy\,dx.
\end{align*}
Then the conclusion follows.
\end{proof}

We end this section with the following technical lemma from \cite[Chapter I, Lemma 4.1]{DiB83}.
\begin{lem}\label{lem:tech}
For $C,b>1$ and $\alpha>0$, let $\{Y_i\}$, $i=0,1,2,\dots,$ be a sequence of positive numbers, satisfying
\begin{align*}
Y_{i+1}\leq C b^i Y_{i}^{1+\alpha}.
\end{align*}
Then the following implication holds:
\begin{align*}
Y_0\leq C^{-\frac{1}{\alpha}}b^{-\frac{1}{\alpha^2}}\quad\implies\quad \lim_{i\rightarrow\infty}Y_i=0.
\end{align*}
\end{lem}

\section{Caccioppoli Estimates}\label{sec:cacc}
In this section, we provide several type of Caccioppoli estimates necessary to prove our main result. We define for any $h>0$, $0<t<T$ and $v\in L^{1}(U_T)$,
\begin{align*}
\llbracket v\rrbracket_{h}(x,t):=\frac{1}{h}\int^t_0 e^{\frac{a-t}{h}}v(x,a)\,da\quad\text{and}\quad \llbracket v\rrbracket_{\overbar{h}}(x,t):=\frac{1}{h}\int^T_t e^{\frac{t-a}{h}}v(x,a)\,da.
\end{align*}
From \cite{Lia21}, we have the following:
\begin{align}\label{eq:identity}
\partial_t\llbracket v\rrbracket_h=\frac{1}{h}\left(v-\llbracket v\rrbracket_h\right)\quad\text{and}\quad\partial_t\llbracket v\rrbracket_{\overbar{h}}=\frac{1}{h}\left(\llbracket v\rrbracket_{\overbar{h}}-v\right).
\end{align}

Now we show the following lemma for the flexibility of choice for test functions.
\begin{lem}\label{lem:sub}
Let $m > 0$ and $u$ be a local weak sub(super)-solution in $\Omega_T$ in the sense of Definition \ref{def:weak}. Then 
for any nonnegative test function $\varphi\in L^{2}(0,T;W^{s,2}_0(\Omega))$ with compact support in $\Omega_T$, we have
\begin{align}\label{eq:moleq}
\begin{split}
&\iint_{\Omega_T}\varphi\partial_t\llbracket u\rrbracket_h\,dx\,dt+\iint_{\Omega_T}\int_{\mathbb{R}^n}\llbracket(\phi(u)(y,t)-\phi(u)(x,t))K(x,y;t)\rrbracket_h(\varphi(y,t)-\varphi(x,t))\,dy\,dx\,dt\\
&\quad \leq(\geq)\int_{\Omega}u(x,0)\cdot\frac{1}{h}\int_{0}^{T} e^{-\frac{a}{h}}\varphi(x,a)\,da\,dx.
\end{split}
\end{align}
\end{lem}
\begin{proof}
If we take the test function $\llbracket\varphi\rrbracket_{\overbar{h}}$ to \cref{eq:main-eq}, then since $\varphi$ has compact support in $\Omega_T$, by Fubini's theorem we have
\begin{align}\label{eq:moleq0}
\begin{split}
&-\iint_{\Omega_T}u\partial_t\llbracket \varphi\rrbracket_{\overbar{h}}\,dx\,dt+\iint_{\Omega_T}\int_{\mathbb{R}^n}\llbracket(\phi(u)(y,t)-\phi(u)(x,t))K(x,y;t)\rrbracket_h(\varphi(y,t)-\varphi(x,t))\,dt\,dx\,dy\\
&\quad \leq(\geq)\int_{\Omega}u(x,0)\cdot\frac{1}{h}\int^{T}_{0} e^{-\frac{a}{h}}\varphi(x,a)\,da\,dx,
\end{split}
\end{align}
where in the above inequality `$\leq$' holds for sub-solution, whereas `$\geq$' holds for super-solution. Then by the computation in \cite[Page 6]{Lia21}, we obtain
\begin{align*}
\iint_{\Omega_T}u\partial_t\llbracket\varphi\rrbracket_{\overbar{h}}\,dx\,dt=-\iint_{\Omega_T}\partial_t\llbracket u\rrbracket_h\varphi\,dx\,dt.
\end{align*}
This completes the proof.
\end{proof}

For $f\in L^{1}_{\text{loc}}(I\times U)$ with an open set $U\subset\mathbb{R}^n$, an open interval $I\subset\mathbb{R}$, and $k\in\mathbb{R}$, we define the truncation operator
\begin{align*}
(f-k)_+:=\max\{f-k,0\}\quad\text{and}\quad (f-k)_-:=\max\{k-f,0\}.
\end{align*}
For $P, Q\in\mathbb{R}$, denoting
\begin{align*}
\g_{\pm}(P,Q) = \pm \int_{Q}^P\left(\phi(R)-\phi(Q)\right)_{\pm}\,dR,
\end{align*}
we prove the following Caccioppoli inequality.
\begin{lem}\label{lem:ee-sing-1}
Let $m > 0$ and $u$ be a local weak sub(super)-solution in $\Omega_T$ in the sense of Definition \ref{def:weak}. Let $R,S>0$. Then there is a constant $C=C(\text{\textbf{data}})>0$ such that for all cylinders $Q_{R,S}(x_0,t_0) := K_R(x_0)\times(t_0-S,t_0]=K\times I \Subset \Omega_T$, for every level $k \in \bb{R}$ and for every smooth nonnegative cutoff function $\xi$ in $Q_{R,S}(x_0,t_0)$ vanishing on $\de K_R(x_0)$ we have 
\begin{align*}
&\esup_{t\in I}\int_K\xi^2\g_{\pm}(u,k)(x,t)\,dx+\int_I\int_K\int_K \min\{\xi^2(x,t),\xi^2(y,t)\}\frac{|(\phi(u)-\phi(k))_{\pm}(x,t)-(\phi(u)-\phi(k))_{\pm}(y,t)|^2}{|x-y|^{n+2s}}\,dx\,dy\,dt \\
&+\int_I\left(\int_K\xi^2(x,t)(\phi(u)-\phi(k))_{\pm}(x,t)\,dx\right)\left(\int_K \frac{(\phi(u)-\phi(k))_{\mp}(y,t)}{|x-y|^{n+2s}}\,dy\right)\,dt\\
&\leq C\int_I\int_K\int_K\max\{(\phi(u)-\phi(k))_{\pm}^2(x,t),(\phi(u)-\phi(k))_{\pm}^2(y,t)\}\frac{|\xi(x,t)-\xi(y,t)|^2}{|x-y|^{n+2s}}\,dx\,dy\,dt + C\int_I\int_K\g_{\pm}(u,k)|\de_t \xi^2|\,dx\,dt\\
&+C\int_I\left(\int_K\xi^2(\phi(u)-\phi(k))_{\pm}(x,t)\,dx\right)\left(\underset{x\in \mbox{supp}(\xi(\cdot,t)); t\in I}{\esup}\int_{\bb{R}^n\setminus K}\frac{(\phi(u)-\phi(k))_{\pm}(y,t)}{|x-y|^{n+2s}}\,dy\right)\,dt+\int_K(\g_{\pm}(u,k)\xi^2)(x,t_0-S)\,dx.
\end{align*}
\end{lem}
\begin{proof}
We deal with subsolutions since the case for supersolutions is analogous. For fixed $0<t_0-S<t_1<t_2<t_0<T$ and $\epsilon>0$ small enough we define the cutoff function in time by
\begin{align}\label{eq:psi}
\psi_{\epsilon}(t):=
\begin{cases}
0\quad&\text{for }0\leq t\leq t_1-\epsilon\quad\text{or}\quad t_2+\epsilon<t\leq T\\
1+\frac{t-t_1}{\epsilon}\quad&\text{for }t_1-\epsilon<t\leq t_1,\\
1\quad&\text{for }t_1<t\leq t_2,\\
1-\frac{t-t_2}{\epsilon}\quad&\text{for }t_2<t\leq t_2+\epsilon.
\end{cases}
\end{align}
Now, we choose in \cref{eq:moleq} the test function
\begin{align}
\varphi_{\epsilon}(x,t)=\xi^2(x,t)\psi_{\epsilon}(t)\left(\phi(u(x,t))-\phi(k)\right)_+\in L^{2}(0,T; W^{s,2}_0(\Omega)).
\end{align}
Then we have
\begin{align}\label{eq:moleq1}
\begin{split}
&\iint_{\Omega_T}(\partial_t\llbracket u\rrbracket_h)\varphi\,dx\,dt+\iint_{\Omega_T}\int_{\mathbb{R}^n}\llbracket(\phi(u)(y,t)-\phi(u)(x,t))K(x,y;t)\rrbracket_h(\varphi(y,t)-\varphi(x,t))\,dy\,dx\,dt\\
&\quad \leq\int_{\Omega}u(x,0)\cdot\frac{1}{h}\int^T_0 e^{-\frac{a}{h}}\varphi(x,a)\,da\,dx.
\end{split}
\end{align}		
		
For the integral in \cref{eq:moleq} containing the time derivative by \cite{Lia21} we have
\begin{align}\label{eq:cacc1}
\begin{split}
&\lim_{\epsilon\rightarrow 0}\liminf_{h\rightarrow 0}\iint_{\Omega_T}(\partial_t\llbracket u\rrbracket_h)\varphi_{\epsilon}\,dx\,dt\\
&\geq -\int_{K}\xi^2(x,t_1)\g_+(u(x,t_1),k)\,dx+\int_{K}\xi^2(x,t_2)\g_+(u(x,t_2),k)\,dx-\iint_{K\times(t_1,t_2)}(\partial_t\xi^2)\g_+(u,k)\,dx\,dt
\end{split}
\end{align}
for any $t_0-S<t_1<t_2<t_0$. Moreover, since $\varphi_{\epsilon}(x,0)\equiv 0$ because of $\psi_{\epsilon}$,
\begin{align}\label{eq:cacc1.1}
\lim_{\epsilon\rightarrow 0}\liminf_{h\rightarrow 0}\int_{\Omega}u(x,0)\cdot\frac{1}{h}\int^{T}_{0}e^{-\frac{a}{h}}\varphi_{\epsilon}(x,a)\,da\,dx=\lim_{\epsilon\rightarrow 0}\int_{\Omega}u(x,0)\varphi_{\epsilon}(x,0)\,dx=0.
\end{align}

Now we focus on the space terms. Applying Fubini's theorem and the argument in \cite[Appendix B]{Lia24} (with the help of \cref{eq:space0}), by symmetry we may write
\begin{align*}
&\lim_{\epsilon\rightarrow 0}\liminf_{h\rightarrow 0}\int_{I}\int_{K}\int_{\mathbb{R}^n}\llbracket(\phi(u)(x,t)-\phi(u)(y,t))K(x,y;t)\rrbracket_h(\varphi_{\epsilon}(x,t)-\varphi_{\epsilon}(y,t))\,dx\,dy\,dt\\
&\quad=\lim_{\epsilon\rightarrow 0}\liminf_{h\rightarrow 0}\int_{I}\int_{K}\int_{\mathbb{R}^n}(\phi(u)(x,t)-\phi(u)(y,t))K(x,y;t)\llbracket(\varphi_{\epsilon}(x,t)-\varphi_{\epsilon}(y,t)\rrbracket_{-h}\,dx\,dy\,dt\\
&\quad=\int_{t_1}^{t_2}\int_{K}\int_{\mathbb{R}^n}(\phi(u)(x,t)-\phi(u)(y,t))(\varphi(x,t)-\varphi(y,t))K(x,y;t)\,dx\,dy\,dt= A+2B,
\end{align*}
where with $\varphi(x,t)=\xi^2(x,t)\left(\phi(u(x,t))-\phi(k)\right)_+$, we write
\begin{align*}
A = \int_{t_1}^{t_2} \int_K\int_K (\phi(u)(x,t)-\phi(u)(y,t)) (\varphi(x,t)-\varphi(y,t))K(x,y,t) \,dx\,dy\,dt
\end{align*}
and
\begin{align*}
B = \int_{t_1}^{t_2}\int_K\int_{\bb{R}^n\setminus K}(\phi(u)(x,t)-\phi(u)(y,t)) \varphi(x,t)K(x,y,t) \,dx\,dy\,dt.
\end{align*}		
We proceed with estimating $A$. For a fixed $t \in [t_1,t_2]$, let $L = \{\phi(u(\cdot,t)) > \phi(k)\} \cap K$. As in \cite[Proposition 8.5]{Coz17}, we obtain
\begin{align*}
&((\phi(u)-\phi(k))_{+}(x,t) - (\phi(u)-\phi(k))_{+}(y,t))((\phi(u)-\phi(k))_+\xi^2(x,t)-(\phi(u)-\phi(k))_+\xi^2(y,t))\\
&\geq c^{-1}((\phi(u)-\phi(k))_{+}(x,t) - (\phi(u)-\phi(k))_{+}(y,t))^2\min\{\xi^2(x,t),\xi^2(y,t)\}\\
&\quad\quad+c^{-1}(\phi(u)-\phi(k))_{-}(y,t)(\phi(u)-\phi(k))_{+}(x,t)\xi^2(x,t)\\
&\quad\quad-c(\phi(u)-\phi(k))^2_+(x,t)|\xi(y,t)-\xi(x,t)|^2
\end{align*}
for some $c\geq 1$. We now estimate the terms in the integral $B$ as follows:
\begin{align*}
-(\phi(u(x,t))-\phi(u(y,t)))(\phi(u)-\phi(k))_{+}(x,t) &\leq (\phi(u(y,t))_+-\phi(u(x,t))_+)(\phi(u)-\phi(k))_{+}(x,t)\\
&\leq (\phi(u)-\phi(k))_{+}(y,t)(\phi(u)-\phi(k))_{+}(x,t).
\end{align*}
Merging above estimates for $A$ and $B$ together with \cref{eq:cacc1} and \cref{eq:cacc1.1} into \cref{eq:moleq1}, the energy estimate follows. 
\end{proof}

Now we prove the following lemma for $m\in(0,1]$.
\begin{lem}\label{lem:ee-sing-2}
Let $m\in (0,1]$ and $u$ be a local weak sub(super)-solution in $\Omega_T$ in the sense of Definition \ref{def:weak}. Also, let $R,S>0$. Then there is a constant $C=C(\text{\textbf{data}})>0$ such that for all cylinders $Q_{R,S}(x_0,t_0) := K_R(x_0)\times(t_0-S,t_0]=K\times I \Subset \Omega_T$, for every level $k \in \bb{R}$ and for every smooth nonnegative cutoff function $\xi$ in $Q_{R,S}(x_0,t_0)$ vanishing on $\de K_R(x_0)$ we have 
\begin{align*}
&\esup_{t \in I}\int_K\xi^2\g_{\pm}(u,k)(x,t)\,dx \\
& \qquad+\int_I\int_K\int_K \min\{\xi^2(x,t),\xi^2(y,t)\}\min\{|u|^{2(m-1)}(x,t),|u|^{2(m-1)}(y,t)\}\frac{|(u-k)_{\pm}(x,t)-(u-k)_{\pm}(y,t)|^2}{|x-y|^{n+2s}}\,dx\,dy\,dt \\
&\leq C\int_I\int_K\int_K\max\{(u-k)_{\pm}^{2m}(x,t),(u-k)_{\pm}^{2m}(y,t)\}\frac{|\xi(x,t)-\xi(y,t)|^2}{|x-y|^{n+2s}}\,dx\,dy\,dt + C\int_I\int_K(u-k)_{\pm}^{m+1}|\de_t \xi^2|\,dx\,dt\\
&\qquad+C\int_I\left(\int_K\xi^2(u-k)_{\pm}^m(x,t)\,dx\right)\left(\underset{x\in \mbox{supp}(\xi(\cdot,t)); t\in I}{\esup}\int_{\bb{R}^n\setminus K}\frac{(\phi(u)-\phi(k))_{\pm}(y,t)}{|x-y|^{n+2s}}\,dy\right)\,dt + \int_K(\xi^2\g_{\pm}(u,k))(x,t_0-S)\,dx.
\end{align*}
\end{lem}
\begin{proof}
We again focus on subsolutions and positive truncations. The case for supersolutions and negative truncations is analogous. We note that since $m \in (0,1]$,
\begin{align*}
(\phi(u) - \phi(k))_+ \leq (u-k)_+^m
\end{align*}
and so in particular, 
\begin{align*}
\g_+(u,k) \leq \int_k^u (R-k)_+^m\,dR = \frac{1}{m+1}(u-k)_+^{m+1}.
\end{align*}
Finally, with $m \in (0,1]$, a case by case analysis and an application of the mean value theorem reveals
\begin{align*}
(\phi(u)-\phi(k))_{+}(x,t)-(\phi(u)-\phi(k))_{+}(y,t) \geq m\min\{|u(x,t)|^{m-1},|u(y,t)|^{m-1}\}((u-k)_{+}(x,t)-(u-k)_{+}(y,t)).
\end{align*}
Using the above estimates in Lemma \ref{lem:ee-sing-1}, the conclusion follows. 
\end{proof}
	
The following is the other version of Caccioppoli inequality.
\begin{lem}\label{lem:ee}
Let $m\geq 1$ and $u:\Omega_T\rightarrow\mathbb{R}$ be a local weak sub(super)-solution in $\Omega_T$ in the sense of Definition \ref{def:weak}. Also, let $R,S>0$. Then there is a constant $C=C(\text{\textbf{data}})>0$ such that for all cylinders $Q_{R,S}(x_0,t_0) := K_R(x_0)\times(t_0-S,t_0]=K\times I\Subset \Omega_T$, for every level $k \in \bb{R}$ and for every smooth cutoff function $\xi$ in $Q_{R,S}(x_0,t_0)$ vanishing on $\de K_R(x_0)$ we have 
\begin{align*}
	&\esup_{t \in I}\int_K\xi^2(u-k)_{\pm}^2(x,t)\,dx \\
	&+ \int_I\int_K\int_K \min\{\xi^2(x,t),\xi^2(y,t)\}\frac{((\phi(u)-\phi(k))_{\pm}(x,t)-(\phi(u)-\phi(k))_{\pm}(y,t))((u-k)_{\pm}(x,t)-(u-k)_{\pm}(y,t))}{|x-y|^{n+2s}}\,dx\,dy\,dt \\
	&+\int_I\left(\int_K\xi^2(x,t)(u-k)_{\pm}(x,t)\,dx\right)\left(\int_K \frac{(\phi(u)-\phi(k))_{\mp}(y,t)}{|x-y|^{n+2s}}\,dy\right)\,dt\\
	&\leq C\int_I\int_K\int_K\max\{|u|^{m-1}(x,t),|u|^{m-1}(y,t)\}\max\{(u-k)_{\pm}^{2}(x,t),(u-k)_{\pm}^{2}(y,t)\}\frac{|\xi(x,t)-\xi(y,t)|^2}{|x-y|^{n+2s}}\,dx\,dy\,dt\\
	&\quad + C\int_I\int_K(u-k)_{\pm}^{2}|\de_t \xi^2|\,dx\,dt\\
	&+C\int_I\left(\int_K\xi^2(u-k)_{\pm}(x,t)\,dx\right)\left(\underset{x\in \mbox{supp}(\xi(\cdot,t)); t\in I}{\esup}\int_{\bb{R}^n\setminus K}\frac{(\phi(u)-\phi(k))_{\pm}(y,t)}{|x-y|^{n+2s}}\,dy\right)\,dt + \int_K(\xi^2(u-k)_{\pm}^2)(x,t_0-S)\,dx.
	\end{align*}	
\end{lem}
\begin{proof}
We deal with subsolutions since the case for supersolutions is analogous. With $\psi_{\epsilon}$ defined in \cref{eq:psi}, we choose a test function
\begin{align*}
\varphi(x,t) = \xi^2(x,t)\psi_{\epsilon}(t)(u-k)_{+}(x,t)\in L^{2}(0,T; W^{s,2}_0(\Omega))
\end{align*}
to \cref{eq:moleq}. Then by the same argument in the proof of Lemma \ref{lem:ee-sing-1} with the help of \cref{eq:space0} and \cref{eq:space}, we obtain
\begin{align}\label{eq:moleq2}
\begin{split}
&\int_{K_R}\xi^2(x,t_2)(u(x,t_2)-k)_+^2\,dx+\int_{t_1}^{t_2}\int_{K}\int_{\mathbb{R}^n}(\phi(u)(x,t)-\phi(u)(y,t))(\varphi(x,t)-\varphi(y,t))K(x,y;t)\,dx\,dy\,dt\\
&\quad \leq\int_{K_R}\xi^2(x,t_1)(u(x,t_1)-k)^2_+\,dx+\iint_{K_{R}\times(t_1,t_2)}(\partial_t\xi^2)(u-k)_+^2\,dx\,dt.
\end{split}
\end{align}
Now we focus on the space terms. By symmetry, we may write 
\begin{align*}
\int_{t_1}^{t_2}\int_{K}\int_{\mathbb{R}^n}(\phi(u)(x,t)-\phi(u)(y,t))(\varphi(x,t)-\varphi(y,t))K(x,y;t)\,dx\,dy\,dt= A+2B,
\end{align*}
where with $\varphi(x,t)=\xi^2(x,t)\left(u-k\right)_+(x,t)$, we write
\begin{align*}
A = \int_{t_1}^{t_2} \int_K\int_K (\phi(u)(x,t)-\phi(u)(y,t)) (\varphi(x,t)-\varphi(y,t))K(x,y,t) \,dx\,dy\,dt
\end{align*}
and
\begin{align*}
B = \int_{t_1}^{t_2}\int_K\int_{\bb{R}^n\setminus K}(\phi(u)(x,t)-\phi(u)(y,t)) \varphi(x,t)K(x,y,t) \,dx\,dy\,dt.
\end{align*}		
We proceed with estimating $A$. For a fixed $t \in I$, let 
\begin{align*}
L = \{\phi(u(\cdot,t)) \geq \phi(k)\} \cap K = \{u(\cdot,t)\geq k\}\cap K.
\end{align*}
Then we have the following three cases: 
\begin{itemize}
\item Case 1: if $x,y \in K\setminus L$ then $\xi(x,t) = \xi(y,t) = 0$ and in particular $\varphi(x)-\varphi(y) = 0$.
\item Case 2: if $x\in L$ and $y \in  K\setminus L$ then 
\begin{align*}
&(\phi(u(x,t))-\phi(u(y,t)))(\varphi(x,t)-\varphi(y,t))\\
&= \left((\phi(u)-\phi(k))_{+}(x,t)+(\phi(u)-\phi(k))_{-}(y,t)\right)(u-k)_{+}(x,t)\xi^2(x,t)\\
&= (\phi(u)-\phi(k))_{+}((u-k)_+\xi^2)(x,t) + (\phi(u)-\phi(k))_{-}(y,t)(u-k)_{+}(x,t)\xi^2(x,t).
\end{align*}
\item Case 3: Suppose that $x,y \in L$. There is no loss of generality in assuming that $u(x,t)>u(y,t)$ because the expression is now symmetric in $x$ and $y$. Note that $u(x,t)>u(y,t)$ also implies that $\phi(u(x,t))>\phi(u(y,t))$. We have
\begin{align*}
&(\phi(u(x,t))-\phi(u(y,t)))(\varphi(x,t)-\varphi(y,t)) \\
&= ((\phi(u)-\phi(k))_{+}(x,t) - (\phi(u)-\phi(k))_{+}(y,t))((u-k)_+\xi^2(x,t)-(u-k)_+\xi^2(y,t)).
\end{align*}
We now deal with two further subcases. 
\begin{itemize}
\item a. if $\xi(x,t) \geq \xi(y,t)$ then
\begin{align*}
((\phi(u)-\phi(k))_{+}(x,t) - (\phi(u)-\phi(k))_{+}(y,t))((u-k)_+\xi^2(x,t)-(u-k)_+\xi^2(y,t))\\
\geq \xi^2(y,t)((\phi(u)-\phi(k))_{+}(x,t) - (\phi(u)-\phi(k))_{+}(y,t))((u-k)_+(x,t)-(u-k)_+(y,t)).
\end{align*}

\item  b. if $\xi(x,t) < \xi(y,t)$ then
\begin{align*}
&((\phi(u)-\phi(k))_{+}(x,t) - (\phi(u)-\phi(k))_{+}(y,t))((u-k)_+\xi^2(x,t)-(u-k)_+\xi^2(y,t))\\
&=((\phi(u)-\phi(k))_{+}(x,t) - (\phi(u)-\phi(k))_{+}(y,t))((u-k)_+(x,t)-(u-k)_+(y,t))\xi^2(y,t)\\
&- ((\phi(u)-\phi(k))_{+}(x,t) - (\phi(u)-\phi(k))_{+}(y,t))(u-k)_+(x,t)(\xi^2(y,t)-\xi^2(x,t)).
\end{align*}
We further estimate
\begin{align*}
&((\phi(u)-\phi(k))_{+}(x,t)-(\phi(u)-\phi(k))_{+}(y,t))(u-k)_+(x,t)(\xi^2(y,t)-\xi^2(x,t))\\
&\leq ((\phi(u)-\phi(k))_{+}(x,t)-(\phi(u)-\phi(k))_{+}(y,t))(u-k)_+(x,t)(\epsilon\xi^2(y,t)+4\epsilon^{-1}|\xi(y,t)-\xi(x,t)|^2)
\end{align*}
for any $\epsilon > 0$.
We set
\begin{align*}
\epsilon = \frac{1}{2}\frac{(u-k)_+(x,t)-(u-k)_+(y,t)}{(u-k)_+(x,t)} > 0
\end{align*}
to obtain
\begin{align*}
&((\phi(u)-\phi(k))_{+}(x,t)-(\phi(u)-\phi(k))_{+}(y,t))(u-k)_+(x,t)(\xi^2(y,t)-\xi^2(x,t))\\
&\leq \frac{1}{2}((\phi(u)-\phi(k))_{+}(x,t)-(\phi(u)-\phi(k))_{+}(y,t))((u-k)_+(x,t)-(u-k)_+(y,t))\xi^2(y,t)\\
&+8(u-k)_+^2(x,t)|\xi(y,t)-\xi(x,t)|^2\frac{(\phi(u)-\phi(k))_{+}(x,t)-(\phi(u)-\phi(k))_{+}(y,t)}{(u-k)_+(x,t)-(u-k)_+(y,t)}.
\end{align*}
\end{itemize}
\end{itemize}
For $m\geq 1$, we note that
\begin{align}\label{eq:mmax}
\frac{(\phi(u)-\phi(k))_{+}(x,t)-(\phi(u)-\phi(k))_{+}(y,t)}{(u-k)_+(x,t)-(u-k)_+(y,t)} = \frac{\phi(u(x,t))-\phi(u(y,t))}{u(x,t)-u(y,t)} \leq m\max\{|u(x,t)|^{m-1},|u(y,t)|^{m-1}\}.
\end{align}

We now estimate the integrand in $B$ as follows:
\begin{align*}
-(\phi(u(x,t))-\phi(u(y,t)))(u-k)_{+}(x,t) &\leq ((\phi(u(y,t)))_+-\phi(u(x,t))_+)(u-k)_{+}(x,t)\\
&\leq (\phi(u)-\phi(k))_{+}(y,t)(u-k)_{+}(x,t).
\end{align*}
Using the above estimate for $A$ and $B$ to \cref{eq:moleq2}, the energy estimate follows.  
\end{proof}

\begin{rem}\label{rem:positive}
We claim that with $m\geq 1$,
\begin{align}\label{eq:less}
((\phi(u)-\phi(k))_{\pm}(x,t)-(\phi(u)-\phi(k))_{\pm}(y,t))((u-k)_{\pm}(x,t)-(u-k)_{\pm}(y,t)) \geq 0.
\end{align}
To do this, we only check 
\begin{align}\label{eq:less'}
((\phi(u)-\phi(k))_{+}(x,t)-(\phi(u)-\phi(k))_{+}(y,t))((u-k)_{+}(x,t)-(u-k)_{+}(y,t)) \geq 0,
\end{align}
since the other case is similarly proved. For a fixed $t \in \mathbb{R}$, let $L = \{\phi(u(\cdot,t)) \geq \phi(k)\} \cap K = \{u(\cdot,t)\geq k\}\cap K$ as in the above lemma. By symmetry, without loss of generality we assume $u(x)\geq u(y)$. If $x,y\in K\setminus L$, then both sides are zero. When $x\in L$ and $y\in K\setminus L$, then since $m\geq 1$,
\begin{align}\label{eq:m>1}
\begin{split}
&(\phi(u)-\phi(k))_{+}(x,t)-(\phi(u)-\phi(k))_{+}(y,t)\\
&=(\phi(u)-\phi(k))_{+}(x,t)\eqsim(|u(x,t)|+|k|)^{m-1}(u-k)_{+}(x,t)\geq |u(x,t)|^{m-1}(u-k)_{+}(x,t).
\end{split}
\end{align}
Also, if $x,y\in L$, then we have
\begin{align*}
(\phi(u)-\phi(k))_{+}(x,t)-(\phi(u)-\phi(k))_{+}(y,t)&=\phi(u)(x,t)-\phi(u)(y,t)\\
&\eqsim(|u(x,t)|+|u(y,t)|)^{m-1}(u(x,t)-u(y,t)).
\end{align*}
Then \cref{eq:less'} is proved.
\end{rem}

\begin{rem}\label{rem:ee-sing}
The above lemma does not hold when $m\in(0,1)$, since the integral involving $\max\{|u|^{m-1}(x,t),|u|^{m-1}(y,t)\}$ in the right-hand side of the estimate may diverge. However, an analogous energy estimate in Lemma \ref{lem:ee} continues to hold for $m \in (0,1)$ under the following restrictions.
\begin{itemize}
\item a. For bounded subsolutions under positive truncation we restrict to $k>0$.
\item b. For bounded supersolutions under negative truncation we restrict to $k<0$.
\end{itemize}
\end{rem}

We justify the above remark as below. First of all, the following observation is necessary.

\begin{lem}\label{lem:L2}
Let $m\in(0,1]$. Then for an open set $U\subset\mathbb{R}^n$, open interval $I_0\subset\mathbb{R}$, and $k\in\mathbb{R}$, 
\begin{align*}
\phi(u)\in L^2_{\text{loc}}(I_0;W^{s,2}_{\text{loc}}(U))\,\,\,\,\text{and}\,\,\,\, u\,\,\,\,\text{is locally bounded}\quad\implies \quad (u-k)_+\in L^2_{\text{loc}}(I_0;W^{s,2}_{\text{loc}}(U)).
\end{align*}
\end{lem}
\begin{proof}
Choose an open interval $I$ an a compact set $K$ such that $I\times K\subset I_0\times U$. Note that if $u(x,t),u(y,t)>k$, then
\begin{align*}
|\phi(u)(x,t)-\phi(u)(y,t)|&=m|u(x,t)-u(y,t)|\int_{0}^{1}|\xi u(x,t)+(1-\xi)u(y,t)|^{m-1}\,d\xi\\
&\geq C(m)(|k|+M)^{m-1} |u(x,t)-u(y,t)|\\
&\eqsim (|k|+M)^{m-1}|(u(x,t)-k)_+-(u(y,t)-k)_+|,
\end{align*}
where we write $M\coloneqq \esup_{I\times K}|u|$.
In addition, we have that if $u(x,t)>k>u(y,t)$, then 
\begin{align*}
|(\phi(u)-\phi(k))_{+}(x,t)-(\phi(u)-\phi(k))_{+}(y,t)|&=|\phi(u)(x,t)-\phi(k)|\\
&=m|u(x,t)-k|\int_{0}^{1}|\xi u(x,t)+(1-\xi)k(y,t)|^{m-1}\,d\xi\\
&\geq C(m) (|k|+M)^{m-1}|u(x,t)-k|\\
&\eqsim (|k|+M)^{m-1}|(u(x,t)-k)_+-(u(y,t)-k)_+|.
\end{align*}
Using this inequalities, we deduce that 
\begin{align*}
\int_{K}\int_{K}\frac{|(u(x,t)-k)_+-(u(y,t)-k)_+|^2}{|x-y|^{n+2s}}\leq C\int_{K}\int_{K}\frac{ |\phi(u)(x,t)-\phi(u)(y,t)|^2}{|x-y|^{n+2s}},
\end{align*}
with the constant $C$ now also depending on $M$; thus $(u-k)_{+}\in L^2(I;W^{s,2}(K))$.
\end{proof}
\begin{rem}\label{rem:no-prob}
    We note that although the estimate in \cref{lem:L2} depends on the local bound for $u$, the estimate itself is never used. We need the lemma only to justify the use of $(u-k)_{\pm}$ as a test function to derive energy estimates when $m\in (0,1)$. 
\end{rem}

Now employing Lemma \ref{lem:sub} and the above lemma to Lemma \ref{lem:ee} together with Definition \ref{def:weak} yields the following estimate.

\begin{lem}\label{lem:ee.m<1}
Let $m\in(0,1]$ and $u:\Omega\times(0,T)\rightarrow\mathbb{R}$ be a bounded, local weak sub(super)-solution in $\Omega_T\subset\mathbb{R}^n\times\mathbb{R}$ in the sense of Definition \ref{def:weak}. Also, let $R,S>0$. Then there is a constant $C=C(\text{\textbf{data}})>0$ such that for all cylinders $Q_{R,S}(x_0,t_0) := K_R(x_0)\times(t_0-S,t_0]=K\times I\Subset \Omega_T$, for every piecewise smooth cutoff function $\xi$ in $Q_{R,S}(x_0,t_0)$ vanishing on $\de K_R(x_0)$ we have 
\begin{align*}
	&\esup_{t \in I}\int_K\xi^2(u-k)_{\pm}^2(x,t)\,dx \\
	&+ \int_I\int_K\int_K \min\{\xi^2(x,t),\xi^2(y,t)\}\min\{|u(x,t)|^{m-1},|u(y,t)|^{m-1}\}\frac{|(u-k)_{\pm}(x,t)-(u-k)_{\pm}(y,t)|^2}{|x-y|^{n+2s}}\,dx\,dy\,dt \\
	&+\int_I\left(\int_K\xi^2(x,t)(u-k)_{\pm}(x,t)\,dx\right)\left(\int_K \frac{(\phi(u)-\phi(k))_{\mp}(y,t)}{|x-y|^{n+2s}}\,dy\right)\,dt\\
	&\leq C\int_I\int_K\int_K\max\{|u|^{m-1}(x,t),|u|^{m-1}(y,t)\}\max\{(u-k)_{\pm}^{2}(x,t),(u-k)_{\pm}^{2}(y,t)\}\frac{|\xi(x,t)-\xi(y,t)|^2}{|x-y|^{n+2s}}\,dx\,dy\,dt\\
	&\quad + C\int_I\int_K(u-k)_{\pm}^{2}|\de_t \xi^2|\,dx\,dt\\
	&+C\int_I\left(\int_K\xi^2(u-k)_{\pm}(x,t)\,dx\right)\left(\underset{x\in \mbox{supp}(\xi(\cdot,t)); t\in I}{\esup}\int_{\bb{R}^n\setminus K}\frac{(\phi(u)-\phi(k))_{\pm}(y,t)}{|x-y|^{n+2s}}\,dy\right)\,dt + \int_K\xi^2(u-k)_{\pm}^2(x,t_0-S)\,dx,
	\end{align*}
with any $k>0$ in case of $(u-k)_+$ and $k<0$ in case of $(u-k)_-$ in the above estimate.	
\end{lem}
\begin{proof}
The proof follows along the same lines as in the previous energy estimates. However, we provide it here for the sake of completeness. We deal with subsolutions since the case for supersolutions is analogous. With the help of Lemma \ref{lem:L2}, we choose a test function
\begin{align*}
\varphi(x,t) = \xi^2(x,t)\psi_{\epsilon}(t)(u-k)_{+}(x,t)\in L^{2}(0,T; W^{s,2}_0(\Omega))
\end{align*}
to \cref{eq:moleq}.

 Then by the same argument in the proof of Lemma \ref{lem:ee-sing-1} with the help of \cref{eq:space0}, we obtain
\begin{align}\label{eq:moleq3}
\begin{split}
&\int_{K_R}\xi^2(x,t_2)(u(x,t_2)-k)_+^2\,dx+\int_{t_1}^{t_2}\int_{K}\int_{\mathbb{R}^n}(\phi(u)(x,t)-\phi(u)(y,t))(\varphi(x,t)-\varphi(y,t))K(x,y;t)\,dx\,dy\,dt\\
&\quad \leq\int_{K_R}\xi^2(x,t_1)(u(x,t_1)-k)^2_+\,dx+\iint_{K_{R}\times(t_1,t_2)}\partial_t\xi^2(u-k)_+^2\,dx\,dt.
\end{split}
\end{align}
Now we focus on the space terms. By symmetry, we may write 
\begin{align*}
\int_{t_1}^{t_2}\int_{K}\int_{\mathbb{R}^n}(\phi(u)(x,t)-\phi(u)(y,t))(\varphi(x,t)-\varphi(y,t))K(x,y;t)\,dx\,dy\,dt= A+2B,
\end{align*}
where with $\varphi(x,t)=\xi^2(x,t)\left(u-k\right)_+(x,t)$, we write
\begin{align*}
A = \int_{t_1}^{t_2} \int_K\int_K (\phi(u)(x,t)-\phi(u)(y,t)) (\varphi(x,t)-\varphi(y,t))K(x,y,t) \,dx\,dy\,dt
\end{align*}
and
\begin{align*}
B = \int_{t_1}^{t_2}\int_K\int_{\bb{R}^n\setminus K}(\phi(u)(x,t)-\phi(u)(y,t)) \varphi(x,t)K(x,y,t) \,dx\,dy\,dt.
\end{align*}

We proceed with estimating $A$. For a fixed $t \in [t_1,t_2]$, let $L = \{\phi(u(\cdot,t)) \geq \phi(k)\} \cap K = \{u(\cdot,t)\geq k\}\cap K|$. Then we have the following three cases: In case of $x,y \in K\setminus L$, and  $x\in L$ and $y \in  K\setminus L$, we can estimate in the same way as the proof of Lemma \ref{lem:ee}. When $x,y \in L$, we can arrive at
\begin{align*}
&(\phi(u(x,t))-\phi(u(y,t)))(\varphi(x,t)-\varphi(y,t))\\
&\leq \frac{1}{2}((\phi(u)-\phi(k))_{+}(x,t)-(\phi(u)-\phi(k))_{+}(y,t))((u-k)_+(x,t)-(u-k)_+(y,t))\xi^2(y,t)\\
&+8(u-k)_+^2(x,t)|\xi(y,t)-\xi(x,t)|^2\frac{(\phi(u)-\phi(k))_{+}(x,t)-(\phi(u)-\phi(k))_{+}(y,t)}{(u-k)_+(x,t)-(u-k)_+(y,t)}.
\end{align*}
We note that
\begin{align*}
\frac{(\phi(u)-\phi(k))_{+}(x,t)-(\phi(u)-\phi(k))_{+}(y,t)}{(u-k)_+(x,t)-(u-k)_+(y,t)} = \frac{\phi(u(x,t))-\phi(u(y,t))}{u(x,t)-u(y,t)} \leq m\max\{|u(x,t)|^{m-1},|u(y,t)|^{m-1}\}.
\end{align*}

We also claim that when $m\in(0,1)$,
\begin{align}\label{eq:less.1}
(\phi(u)-\phi(k))_{+}(x,t)-(\phi(u)-\phi(k))_{+}(y,t) \geq C^{-1}\min\{|u(x,t)|^{m-1},|u(y,t)|^{m-1}\}((u-k)_{+}(x,t)-(u-k)_{+}(y,t)).
\end{align}
Indeed, if $x,y\in K\setminus L$, then the both sides are zero. When $x\in L$ and $y\in K\setminus L$, then since $m\in(0,1)$,
\begin{align*}
(\phi(u)-\phi(k))_{+}(x,t)-(\phi(u)-\phi(k))_{+}(y,t)&=(\phi(u)-\phi(k))_{+}(x,t)\\
&\eqsim(|u(x,t)|+k)^{m-1}(u-k)_{+}(x,t)\\
&\geq\min\{|u(x,t)|^{m-1},|u(y,t)|^{m-1}\}((u-k)_{+}(x,t)-(u-k)_{+}(y,t)).
\end{align*}
Also, if $x,y\in L$, then without loss of generosity we assume $u(x)\geq u(y)$. Then we have
\begin{align*}
(\phi(u)-\phi(k))_{+}(x,t)-(\phi(u)-\phi(k))_{+}(y,t)&=\phi(u)(x,t)-\phi(u)(y,t)\\
&\eqsim|u(x,t)|^{m-1}(u(x,t)-u(y,t))\\
&=\min\{|u(x,t)|^{m-1},|u(y,t)|^{m-1}\}((u-k)_{+}(x,t)-(u-k)_{+}(y,t)).
\end{align*}
Then \cref{eq:less.1} is proved. We now estimate the integrand in $B$ as follows:
\begin{align*}
-(\phi(u(x,t))-\phi(u(y,t)))(u-k)_{+}(x,t) &\leq ((\phi(u(y,t)))_+-\phi(u(x,t))_+)(u-k)_{+}(x,t)\\
&\leq (\phi(u)-\phi(k))_{+}(y,t)(u-k)_{+}(x,t).
\end{align*}
Using the above estimate for $A$ and $B$ to \cref{eq:moleq3}, the energy estimate follows.
\end{proof}


\section{De Giorgi Lemmas}\label{sec:De}
For $R>0$, $T_1,T_2\in\mathbb{R}$ and $x_0\in\mathbb{R}^n$, let $Q = K_R(x_0)\times(T_1,T_2] \subset \Omega_T$. To incorporate the singularity/degeneracy of the equation \cref{eq:main-eq} in our analysis we need to work in a geometry intrinsic to the equation. We do this via working in cylinders which are scaled in space for the fast diffusion case and with cylinders scaled in time for the porous media case. The scaling factor will be chosen depending on $|u|$ - the quantity responsible for the  singularity/degeneracy. 

For a scaling factor $\theta>0$ we define
\begin{align*}
(x_0,t_0)+\mathfrak{Q}_{\rho}(\vartheta) = K_{\vartheta\rho}(x_0)\times(t_0-\rho^{2s},t_0]
\end{align*}
and
\begin{align*}
(x_0,t_0)+Q_{\rho}(\theta) = K_{\rho}(x_0)\times(t_0-\theta\rho^{2s},t_0]
\end{align*}
with $(x_0,t_0)\in\mathbb{R}^{n+1}$ and $\rho>0$. We omit the vertex $(x_0,t_0)$ when it is clear from the context.


With $(x_0,t_0)\in\mathbb{R}^{n+1}$ and $\rho,\theta>0$, we work with a cylinder
\begin{align*}
(x_0,t_0)+Q_{\rho}(\theta) \subset Q \subset \Omega_T
\end{align*}
and let $M>0$ be a number such that
\begin{align}\label{eq:2M.sup}
2M \geq \esup_{Q} |u|.
\end{align}
Before giving the main result, we prove the following lemma.
\begin{lem}\label{lem:est}
Let $m\geq 1$ and \cref{eq:2M.sup} hold. With either
\begin{align}\label{eq:uM}
\widetilde{u}:=\max\left\{u,M/4\right\}\quad\text{and a number}\quad k\in\left[M/2,2M\right]
\end{align}
or 
\begin{align}\label{eq:uM2}
\widetilde{u}:=\min\left\{u,-M/4\right\}\quad\text{and a number}\quad k\in\left[-2M,-M/2\right],
\end{align}
we have
\begin{align}\label{eq:est0}
\begin{split}
&M^{m-1}|(\widetilde{u}-k)_{\pm}(x,t)-(\widetilde{u}-k)_{\pm}(y,t)|^2\\
&\quad\leq 4^m((\phi(u)-\phi(k))_{\pm}(x,t)-(\phi(u)-\phi(k))_{\pm}(y,t))((u-k)_{\pm}(x,t)-(u-k)_{\pm}(y,t)).
\end{split}
\end{align}
\end{lem}
\begin{proof}
We only prove that
\begin{align}\label{eq:est}
\begin{split}
&M^{m-1}|(\widetilde{u}-k)_{-}(x,t)-(\widetilde{u}-k)_{-}(y,t)|^2\\
&\quad\leq 4^m((\phi(u)-\phi(k))_{-}(x,t)-(\phi(u)-\phi(k))_{-}(y,t))((u-k)_{-}(x,t)-(u-k)_{-}(y,t))
\end{split}
\end{align}
in case of \cref{eq:uM}, since the other three cases are proved similarly. By the symmetry of both sides of \cref{eq:est} and the monotonicity of $\phi$, we can always assume
\begin{align*}
u(y,t)\geq u(x,t).
\end{align*}
Now we divide the proof into three cases.

\textbf{Case 1: When $u(x,t),u(y,t)\leq\frac{M}{4}$. } In this case, $\widetilde{u}(x,t)=\widetilde{u}(y,t)\equiv\frac{M}{4}$. Then we have
\begin{align*}
|(\widetilde{u}-k)_-(x,t)-(\widetilde{u}-k)_-(y,t)|\equiv 0
\end{align*}
so that the left-hand side of \cref{eq:est} is zero. Also, by \cref{eq:less} in Remark \ref{rem:positive}, the right-hand side of \cref{eq:est} is nonnegative, thus \cref{eq:est} is true.

\textbf{Case 2: When $u(x,t),u(y,t)>\frac{M}{4}$. } In this case, we guarantee $u(x,t),u(y,t)>0$. Since $\frac{M}{4}<u(x,t)\leq 2M$, $\widetilde{u}(x,t)\equiv u(x,t)$, and $\widetilde{u}(y,t)\equiv u(y,t)$, we have
\begin{align}\label{eq:eqsim}
M^{m-1}|(\widetilde{u}-k)_-(x,t)-(\widetilde{u}-k)_-(y,t)|^2\eqsim |u(x,t)|^{m-1}((u-k)_-(x,t)-(u-k)_-(y,t))^2
\end{align}
with the implicit constant $c=4^m$. Now we claim the following:
\begin{align}\label{eq:claim}
|u(x,t)|^{m-1}|(u-k)_-(x,t)-(u-k)_-(y,t)|\leq (\phi(u)-\phi(k))_{-}(x,t)-(\phi(u)-\phi(k))_{-}(y,t).
\end{align}
To show the above claim, we divide the case into three subcases.
\begin{itemize}
\item[] \textbf{Subcase 1: When $u(x,t),u(y,t)>k$. } In this case, both sides of \cref{eq:claim} is zero, so \cref{eq:claim} is trivially true.
\item[] \textbf{Subcase 2: When $u(y,t)>k$ and $u(x,t)\leq k$. } Note that $|u(x,t)|\leq k$ from $u(x,t)>0$. Then together with $(u-k)_-(y,t)\equiv(\phi(u)-\phi(k))_-(y,t)\equiv 0$ and $m\geq 1$, we estimate
\begin{align*}
|u(x,t)|^{m-1}|(u-k)_-(x,t)-(u-k)_-(y,t)|&=|u(x,t)|^{m-1}(k-u(x,t))\\
&=k|u(x,t)|^{m-1}-\phi(u)(x,t)\\
&\leq \phi(k)-\phi(u)(x,t)\\
&=(\phi(u)-\phi(k))_-(x,t)-(\phi(u)-\phi(k))_-(y,t).
\end{align*}
\item[] \textbf{Subcase 3: When $u(x,t),u(y,t)\leq k$. } In this case, recalling $u(y,t)\geq u(x,t)\geq 0$ and $m\geq 1$, we have
\begin{align*}
|u(x,t)|^{m-1}|(u-k)_-(x,t)-(u-k)_-(y,t)|&=|u(x,t)|^{m-1}(u(y,t)-u(x,t))\\
&\leq\phi(u)(y,t)-\phi(u)(x,t)\\
&=(\phi(k)-\phi(u)(x,t))-(\phi(k)-\phi(u)(y,t))\\
&=(\phi(u)-\phi(k))_-(x,t)-(\phi(u)-\phi(k))_-(y,t).
\end{align*}
\end{itemize}
Considering \textbf{Subcase 1}--\textbf{Subcase 3}, we obtain \cref{eq:claim}. Now with \cref{eq:eqsim}, we get
\begin{align*}
M^{m-1}|(\widetilde{u}-k)_-(x,t)-(\widetilde{u}-k)_-(y,t)|^2&\leq 4^m |u(x,t)|^{m-1}|(u-k)_-(x,t)-(u-k)_-(y,t)|^2\\
&\leq 4^m((\phi(u)-\phi(k))_{-}(x,t)-(\phi(u)-\phi(k))_{-}(y,t))|(u-k)_-(x,t)-(u-k)_-(y,t)|\\
&=4^m((\phi(u)-\phi(k))_{-}(x,t)-(\phi(u)-\phi(k))_{-}(y,t))((u-k)_-(x,t)-(u-k)_-(y,t)),
\end{align*}
where for the last equality we have used $(u-k)_-(x,t)-(u-k)_-(y,t)\geq 0$ since $u(x,t)\leq u(y,t)$.

\textbf{Case 3: When $u(x,t)\leq\frac{M}{4}$ and $u(y,t)>\frac{M}{4}$. } In this case, note that $u(x,t)<0$ is possible. First, observe that
\begin{align*}
u(y,t)\equiv\widetilde{u}(y,t)\quad\text{and}\quad \widetilde{u}(x,t)=\tfrac{M}{4}
\end{align*}
and so
\begin{align*}
M^{m-1}|(\widetilde{u}-k)_-(x,t)-(\widetilde{u}-k)_-(y,t)|^2=M^{m-1}|(k-\tfrac{M}{4})-(u-k)_-(y,t)|^2.
\end{align*}
Now we consider two subcases.
\begin{itemize}
\item[] \textbf{Subcase 1: When $u(y,t)\geq k$. } In this case, since $(u-k)_-(y,t)\equiv 0$ and $k\geq \frac{M}{4}\geq u(x,t)$, together with $m\geq 1$ we obtain
\begin{align*}
M^{m-1}|(k-\tfrac{M}{4})-(u-k)_-(y,t)|^2&\leq 4^m k^{m-1}(k-\tfrac{M}{4})(k-\tfrac{M}{4})\\
&=4^m (\phi(k)-k^{m-1}\tfrac{M}{4})(k-\tfrac{M}{4})\\
&\leq 4^m (\phi(k)-\phi(u)(x,t))(k-u(x,t))\\
&=4^m ((\phi(u)-\phi(k))_-(x,t)-(\phi(u)-\phi(k))_-(y,t))((u-k)_-(x,t)-(u-k)_-(y,t)).
\end{align*}

\item[] \textbf{Subcase 2: When $\frac{M}{4}<u(y,t)<k$. } Then together with $u(x,t)\leq\frac{M}{4}$ and $m\geq 1$, we estimate
\begin{align*}
M^{m-1}|(k-\tfrac{M}{4})-(u-k)_-(y,t)|^2&=4^{m-1}(\tfrac{M}{4})^{m-1}(u(y,t)-\tfrac{M}{4})^2\\
&=4^{m-1}\left((\tfrac{M}{4})^{m-1}u(y,t)-(\tfrac{M}{4})^{m-1}\tfrac{M}{4}\right)(u(y,t)-\tfrac{M}{4})\\
&\leq 4^{m-1}\left(\phi(u)(y,t)-\phi(u)(x,t)\right)(u(y,t)-u(x,t))\\
&=4^{m-1}((\phi(u)-\phi(k))_-(x,t)-(\phi(u)-\phi(k))_-(y,t))((u-k)_-(x,t)-(u-k)_-(y,t)).
\end{align*}
\end{itemize}
Considering \textbf{Subcase 1} and \textbf{Subcase 2}, we obtain \cref{eq:est} in \textbf{Case 3}. Therefore, combining \textbf{Case 1}--\textbf{Case 3}, we get \cref{eq:est} in any case.
\end{proof}

We now prove two variants of De Giorgi's lemma. 

\begin{lem}\label{lem:deG-general}
For $m\geq 1$, let $u$ be a local weak super(sub)-solution in $\Omega_T$ in the sense of Definition \ref{def:weak}. There exists a constant $0<\nu<1$ depending only in $M,\theta$ and the \textbf{data} such that if
\begin{align*}
|\{\mp u \leq M\} \cap (x_0,t_0)+Q_{\rho}(\theta)| \leq \nu|Q_{\rho}(\theta)|\quad\text{and}\quad \left(\frac{\rho}{R}\right)^{\frac{2s}{m}}\tl(u_{\mp},Q) \leq {M}
\end{align*}		
hold, then we have
\begin{align*}
\mp u \geq \frac{1}{2}M \quad \mbox{ a.e. in } \quad (x_0,t_0)+Q_{\tfrac{\rho}{2}}(\theta). 
\end{align*}
Furthermore,
\begin{align*}
\nu = C \theta^{-1} M^{1-m} \left(1+\theta^{-1} M^{1-m}\right)^{-\frac{n+2s}{2s}},
\end{align*}
where $C>0$ is a constant that depends only on the \textbf{data}.
\end{lem}

\begin{proof}
Without loss of generality we assume that $(x_0,t_0) = (0,0)$ and work with super-solutions; the case for sub-solutions being analogous. The energy estimate Lemma \ref{lem:ee} becomes:
\begin{align*}
	&\esup_{t \in (-\theta\rho^{2s},0)}\int_{K_{\rho}}\xi^2(u-k)_{-}^2(x,t)\,dx \\
	&+ \int^{0}_{-\theta\rho^{2s}}\int_{K_{\rho}}\int_{K_{\rho}} \min\{\xi^2(x,t),\xi^2(y,t)\}\frac{((\phi(u)-\phi(k))_{-}(x,t)-(\phi(u)-\phi(k))_{-}(y,t))((u-k)_{-}(x,t)-(u-k)_{-}(y,t))}{|x-y|^{n+2s}}\,dx\,dy\,dt \\
	&\leq C\int^{0}_{-\theta\rho^{2s}}\int_{K_{\rho}}\int_{K_{\rho}}\max\{|u|^{m-1}(x,t),|u|^{m-1}(y,t)\}\max\{(u-k)_{-}^{2}(x,t),(u-k)_{-}^{2}(y,t)\}\frac{|\xi(x,t)-\xi(y,t)|^2}{|x-y|^{n+2s}}\,dx\,dy\,dt\\
	&\quad+C\int^{0}_{-\theta\rho^{2s}}\int_{K_{\rho}}(u-k)_{-}^{2}|\de_t \xi^2|\,dx\,dt+\int_{K_{\rho}}\xi^2(u-k)_{-}^2(x,-\theta\rho^{2s})\,dx\\
	&\quad+C\int^{0}_{-\theta\rho^{2s}}\int_{K_{\rho}}\xi^2(u-k)_{-}(x,t)\,dx\,dt\left(\underset{x\in \mbox{supp}(\xi(\cdot,t)); t\in (-\theta\rho^{2s},0)}{\esup}\int_{\bb{R}^n\setminus K_{\rho}}\frac{(\phi(u)-\phi(k))_{-}(y,t)}{|x-y|^{n+2s}}\,dy\right).
\end{align*}	
For $i = 0, 1, \ldots$ we set
\begin{align*}
k_i = \frac{M}{2}+\frac{M}{2^{i+1}}, \quad \rho_i = \rho+\frac{\rho}{2^i}, \quad K_i = K_{\rho_i}, \quad Q_i = K_i\times (-\theta\rho_i^{2s},0],	
\end{align*}
and
\begin{align*}
\widetilde{k}_i = \frac{k_i+k_{i+1}}{2}, \quad \widetilde{\rho}_i = \frac{\rho_i+\rho_{i+1}}{2}, \quad \widetilde{K}_i = K_{\widetilde{\rho}_i}, \quad \widetilde{Q}_i = \widetilde{K}_i\times (-\theta\widetilde{\rho}_i^{2s},0].
\end{align*}
For the cutoff function in the energy estimate, we work with a cut-off function $0\leq \xi \leq 1$ such that
\begin{align*}
|\nabla \xi| \leq C\frac{2^i}{\rho}, \quad  |\de_t \xi| \leq C\frac{4^{is}}{\theta\rho^{2s}} \quad \mbox{ and } \quad \xi = 1 \mbox{ on } \widetilde{Q}_{i}
\end{align*}
with the support of $\xi$ chosen to be away from $\de_P Q_i$ such that for any $S\in \text{supp}(\xi)$ and $P \in \bb{R}^n\setminus Q_i$ we have
\begin{align}\label{eq:SPP}
\frac{|S-P|}{|P|} \geq \frac{C}{2^i}.
\end{align}
We set 
{\begin{align*}
A_i = {\{u<k_i\} \cap Q_i}.
\end{align*}}
The choice of the support of $\xi$ ensures that
\begin{align*}
\int_{K_i}\xi^2(u-k_i)_{-}^2(x,-\theta\rho_i^{2s})\,dx = 0
\end{align*}
and
\begin{align*}
\underset{x\in \mbox{supp}(\xi(\cdot,t)); t\in (-\theta\rho_{i}^{2s},0)}{\esup}\int_{\bb{R}^n\setminus K_i}\frac{(\phi(u)-\phi(k_i))_{-}(y,t)}{|x-y|^{n+2s}}\,dy \leq C2^{i(n+2s)}\underset{t\in (-\theta\rho_{i}^{2s},0)}{\esup}\int_{\bb{R}^n\setminus K_i}\frac{(\phi(u)-\phi(k_i))_{-}(y,t)}{|y|^{n+2s}}\,dy.
\end{align*}
In particular, we can estimate the nonlocal term as follows:
\begin{align*}
	&\int_{-\theta\rho_{i}^{2s}}^0\int_{K_i}\xi^2(u-k_i)_{-}(x,t)\,dx\,dt\left(\underset{x\in \mbox{supp}(\xi(\cdot,t)); t\in (-\theta\rho_{i}^{2s},0)}{\esup}\int_{\bb{R}^n\setminus K_i}\frac{(\phi(u)-\phi(k_i))_{-}(y,t)}{|x-y|^{n+2s}}\,dy\right)\\
	&\leq C2^{i(n+2s)}\int_{-\theta\rho_{i}^{2s}}^0\int_{K_i}(u-k_i)_{-}(x,t)\,dx\,dt\left(\underset{t\in (-\theta\rho_{i}^{2s},0)}{\esup}\int_{\bb{R}^n\setminus K_i}\frac{(\phi(u)-\phi(k_i))_{-}(y,t)}{|y|^{n+2s}}\,dy\right)\\
	&\leq C2^{i(n+2s)}|A_i|M\left(\frac{CM^m}{\rho^{2s}}+\underset{t\in (-\theta\rho_{i}^{2s},0)}{\esup}\int_{\bb{R}^n\setminus K_R}\frac{(\phi(u)-\phi(k_i))_{-}(y,t)}{|y|^{n+2s}}\,dy\right)\\
	&\leq C2^{i(n+2s)}|A_i|M\left(\frac{M^{m}}{\rho^{2s}}+\left(\frac{\rho}{R}\right)^{2s}\frac{M^{m}}{\rho^{2s}}+\left(\frac{\rho}{R}\right)^{2s}\frac{\tl^m(u_{-},Q)}{\rho^{2s}}\right),
\end{align*}
where we were able to replace the integral over $\bb{R}^n\setminus K_i$ with $\bb{R}^n\setminus K_R$ because of the choice of $M$ and the by the definition of $k_i$. Therefore with the assumption
\begin{align*}
\left(\frac{\rho}{R}\right)^{\frac{2s}{m}}\tl(u_{-},Q) \leq M
\end{align*}
and the choice of cylinders, we get
\begin{align*}
\int_{-\theta\rho_{i}^{2s}}^0\int_{K_i}\xi^2(u-k_i)_{-}(x,t)\,dx\,dt\left(\underset{x\in \mbox{supp}(\xi(\cdot,t)); t\in (-\theta\rho_{i}^{2s},0)}{\esup}\int_{\bb{R}^n\setminus K_i}\frac{(\phi(u)-\phi(k_i))_{-}(y,t)}{|x-y|^{n+2s}}\,dy\right) \leq C2^{i(n+2s)}|A_i|\frac{M^{m+1}}{\rho^{2s}}.
\end{align*}
Recalling that $2M \geq |u|$ locally and the properties of the test function, we estimate the remaining terms on the right hand side as follows:
\begin{align*}
	&\int^{0}_{-\theta\rho^{2s}_i}\int_{K_{i}}\int_{K_{i}}\max\{|u|^{m-1}(x,t),|u|^{m-1}(y,t)\}\max\{(u-k_i)_{-}^{2}(x,t),(u-k_i)_{-}^{2}(y,t)\}\frac{|\xi(x,t)-\xi(y,t)|^2}{|x-y|^{n+2s}}\,dx\,dy\,dt\\
	&\quad\leq C\frac{4^{i}}{\rho^{2s}}M^{m+1}|A_i|	
\end{align*}
and
\begin{align*}
\int_{-\theta\rho_i^{2s}}^0\int_{K_i}(u-k_i)_{-}^{2}|\de_t \xi^2|\,dx\,dt \leq C\frac{4^{is}}{\theta\rho^{2s}}M^{2}|A_i|.
\end{align*}

For the left land side of the energy estimate we estimate from below over the smaller region $\widetilde{K}_{\rho}$ in space and $(-\theta\widetilde{\rho}^{2s},0)$ in time where the test functions are identically unity. Before that, we consider $\widetilde{u}$ defined in \cref{eq:uM} of Lemma \ref{lem:est}. Then by Lemma \ref{lem:est} we have
\begin{align*}
&M^{m-1}\int_{-\theta\widetilde{\rho}_i^{2s}}^0\int_{\widetilde{K}_i}\int_{\widetilde{K}_i}\frac{|(\widetilde{u}-\widetilde{k}_i)_{-}(x,t)-(\widetilde{u}-\widetilde{k}_i)_{-}(y,t)|^2}{|x-y|^{n+2s}}\,dx\,dy\,dt\\
&\leq C\int^{0}_{-\theta\rho_{i}^{2s}}\int_{K_{i}}\int_{K_{i}}\frac{((\phi(u)-\phi(\widetilde{k}_i))_{-}(x,t)-(\phi(u)-\phi(\widetilde{k}_i))_{-}(y,t))((u-\widetilde{k}_i)_-(x,t)-(u-\widetilde{k}_i)_-(y,t))}{|x-y|^{n+2s}}\,dx\,dy\,dt 
\end{align*}
using that $|u| \leq 2M$, $\widetilde{k}_i\in[\tfrac{M}{2},M]$ and $m-1\geq 0$. Putting together the above estimates yields
\begin{align}\label{eq:caccM}
\begin{split}
&\esup_{t \in (-\theta\widetilde{\rho}_i^{2s},0)}\int_{\widetilde{K}_i}(\widetilde{u}-\widetilde{k}_i)_{-}^2\,dx+ M^{m-1}\int_{-\theta\widetilde{\rho}_i^{2s}}^0\int_{\widetilde{K}_i}\int_{\widetilde{K}_i}\frac{|(\widetilde{u}-\widetilde{k}_i)_{-}(x,t)-(\widetilde{u}-\widetilde{k}_i)_{-}(y,t)|^2}{|x-y|^{n+2s}}\,dx\,dy\,dt \\
&\leq C\left(\frac{4^{i}}{\rho^{2s}}M^{m+1}+\frac{4^{is}}{\theta\rho^{2s}}M^{2}+\frac{2^{i(n+4)}}{\rho^{2s}}M^{m+1}\right)|A_i|.
\end{split}
\end{align}

Let $0\leq \psi \leq 1$ be a cutoff function such that $\psi\equiv 0$ on $\partial_P \widetilde{Q}_i$ and $\psi = 1$ on $Q_{i+1}$ with $|\nabla\psi| \leq 2^i\rho^{-1}$.  We apply H\"older's inequality and Sobolev embedding \cref{eq:SP1} to get
\begin{align*}
	\left(\frac{M}{2^{i+3}}\right)^{2}|{A}_{i+1}|&\leq \iint_{\widetilde{Q}_i}(\widetilde{u}-\widetilde{k}_i)_{-}^2\psi^2 \,dx\,dt \\
	&\leq\left(\iint_{\widetilde{Q}_i}((\widetilde{u}-\widetilde{k}_i)\psi)_{-}^{2\frac{n+2s}{n}}\,dx\,dt\right)^{\frac{n}{n+2s}}|A_i|^{\frac{2s}{n+2s}}\\
	&\leq \left(\int_{-\theta\widetilde{\rho}_i^{2s}}^0\int_{\widetilde{K}_i}\int_{\widetilde{K}_i} \frac{|((\widetilde{u}-\widetilde{k}_i)\psi)_{-}(x,t)-((\widetilde{u}-\widetilde{k}_i)\psi)_{-}(y,t)|^2}{|x-y|^{n+2s}} \,dx\,dy\,dt \right)^{\frac{n}{n+2s}}\\
	&\quad\quad\times \left(\esup_{t\in (-\theta\widetilde{\rho}_i^{2s},0)}\int_{\widetilde{K}_i} |((\widetilde{u}-\widetilde{k}_i)\psi)_{-}(x,t)|^2\,dx\right)^{\frac{2s}{n+2s}}{|{A}_i|^{\frac{2s}{n+2s}}}.
\end{align*}
Using $|\nabla \xi|\leq C2^i/\rho$ and the triangle inequality, we have
\begin{align*}
&\int_{-\theta\widetilde{\rho}_i^{2s}}^0\int_{\widetilde{K}_i}\int_{\widetilde{K}_i} \frac{|((\widetilde{u}-\widetilde{k}_i)\psi)_{-}(x,t)-((\widetilde{u}-\widetilde{k}_i)\psi)_{-}(y,t)|^2}{|x-y|^{n+2s}} \,dx\,dy\,dt\\
&\quad\leq \int_{-\theta\widetilde{\rho}_i^{2s}}^0\int_{\widetilde{K}_i}\int_{\widetilde{K}_i} \frac{|(\widetilde{u}-\widetilde{k}_i)_{-}(x,t)-(\widetilde{u}-\widetilde{k}_i)_{-}(y,t)|^2}{|x-y|^{n+2s}} \,dx\,dy\,dt+\dfrac{4^{i}}{\rho^{2s}}\int_{-\theta\widetilde{\rho}_i^{2s}}^0\int_{\widetilde{K}_i}(\widetilde{u}-\widetilde{k}_i)_{-}^{2}(x,t)\,dx\,dt.
\end{align*}
From \cref{eq:caccM} we get
\begin{align*}
	\int_{-\theta\widetilde{\rho}_i^{2s}}^0\int_{\widetilde{K}_i}\int_{\widetilde{K}_i} \frac{|(\widetilde{u}-\widetilde{k}_i)_{-}(x,t)-(\widetilde{u}-\widetilde{k}_i)_{-}(y,t)|^2}{|x-y|^{n+2s}} \,dx\,dy\,dt \leq C\left(\frac{4^{i}}{\rho^{2s}}M^{2}+\frac{4^{is}}{\theta\rho^{2s}}M^{3-m}+\frac{2^{i(n+4)}}{\rho^{2s}}M^{2}\right)|A_i|.
\end{align*}
We also have
\begin{align*}
	\frac{4^i}{\rho^{2s}}\iint_{\widetilde{Q}_i}(\widetilde{u}-\widetilde{k}_i)_{-}^2\,dx\,dt&\leq C \frac{4^i}{\rho^{2s}}M^2|A_i|.
\end{align*}
For the time term, using $0 \leq \psi \leq 1$, we get from above that
\begin{align*}
	\esup_{t\in (-\theta\widetilde{\rho}_i^{2s},0)}\int_{\widetilde{K}_i} |((\widetilde{u}-\widetilde{k}_i)\psi)_{-}(x,t)|^2\,dx\leq C\left(\frac{4^{i}}{\rho^{2s}}M^{m+1}+\frac{4^{is}}{\theta\rho^{2s}}M^{2}+\frac{2^{i(n+4)}}{\rho^{2s}}M^{m+1}\right)|A_i|.
\end{align*}
Setting
\begin{align}\label{eq:De1}
Y_i = \frac{|A_i|}{|Q_i|}
\end{align}
and simplifying we get
\begin{align}\label{eq:De2}
Y_{i+1} \leq Cb^i\left(\frac{M^{1-m}}{\theta}\right)^{-\frac{2s}{n+2s}}\left(1+\frac{M^{1-m}}{\theta}\right)Y_i^{1+\frac{2s}{n+2s}}.
\end{align}
for some $b>1$. Therefore, by Lemma \ref{lem:tech}, $Y_i \rightarrow 0$ provided
\begin{align}\label{eq:De3}
Y_0 \leq C \left(\frac{M^{1-m}}{\theta}\right) \left(1+\frac{M^{1-m}}{\theta}\right)^{-\frac{n+2s}{2s}}=:\nu.
\end{align}
The conclusion follows.  
\end{proof}

\begin{lem}\label{lem:deG-general'}
For $0<m\leq 1$, let $u$ be a local weak super(sub)-solution in $\Omega_T$ in the sense of Definition \ref{def:weak}. There exists a constant $0<\nu'<1$ depending only in $M,\vartheta$ and the \textbf{data} such that if
\begin{align*}
|\{\mp u \leq M\} \cap (x_0,t_0)+\mathfrak{Q}_{\rho}(\vartheta)| \leq \nu'|\mathfrak{Q}_{\rho}(\vartheta)|
\end{align*}		
and
\begin{align*}
\left(\frac{\vartheta\rho}{R}\right)^{\frac{2s}{m}}\tl(u_{\mp},Q) \leq {M}
\end{align*}
hold, then we have
\begin{align*}
\mp u \geq \frac{1}{2}M \quad \mbox{ a.e. in } \quad (x_0,t_0)+\mathfrak{Q}_{\tfrac{\rho}{2}}(\vartheta). 
\end{align*}
Furthermore,
\begin{align*}
\nu' = C \vartheta^{2s}M^{1-m} \left(1+\vartheta^{2s}M^{1-m}\right)^{-\frac{n+2s}{2s}},
\end{align*}
where $C>0$ is a constant that depends only on the \textbf{data}.
\end{lem}

\begin{proof}
Without loss of generality we assume that $(x_0,t_0) = (0,0)$ and work with super-solutions; the case for sub-solutions being analogous. We note that for any $k>\widetilde{k}$ we have 
\begin{align*}
\int_u^k (\phi(s)-\phi(k))_{-}\,ds \geq \int_u^{\widetilde{k}} (\phi(s)-\phi(k))_{-}\,ds \geq (\phi(k)-\phi(\widetilde{k}))(u-\widetilde{k})_{-}.
\end{align*}
In particular, the energy estimate Lemma \ref{lem:ee-sing-2} becomes:
\begin{align*}
&(\phi(k)-\phi(\widetilde{k}))\esup_{t \in (-\rho^{2s},0)}\int_{K_{\vartheta\rho}}\xi^2(u-\widetilde{k})_{-}\,dx \\
& \qquad+ \int_{-\rho^{2s}}^0\int_{K_{\vartheta\rho}}\int_{K_{\vartheta\rho}} \min\{\xi^2(x,t),\xi^2(y,t)\}\min\{|u(x,t)|^{2(m-1)},|u(y,t)|^{2(m-1)}\}\frac{|(u-k)_{-}(x,t)-(u-k)_{-}(y,t)|^2}{|x-y|^{n+2s}}\,dx\,dy\,dt \\
&\leq C\int_{-\rho^{2s}}^0\int_{K_{\vartheta\rho}}\int_{K_{\vartheta\rho}}\max\{(u-k)_{-}^{2m}(x,t),(u-k)_{-}^{2m}(y,t)\}\frac{|\xi(x,t)-\xi(y,t)|^2}{|x-y|^{n+2s}}\,dx\,dy\,dt + \int_{-\rho^{2s}}^0\int_{K_{\vartheta\rho}}(u-k)_{-}^{m+1}|\de_t \xi^2|\,dx\,dt\\
&+\int_{-\rho^{2s}}^0\int_{K_{\vartheta\rho}}\xi^2(u-k)_{-}^m(x,t)\,dx\,dt\left(\underset{x\in \mbox{supp}(\xi(\cdot,t)); t\in (-\rho^{2s},0)}{\esup}\int_{\bb{R}^n\setminus K_{\vartheta\rho}}\frac{(\phi(u)-\phi(k))_{-}(y,t)}{|x-y|^{n+2s}}\,dy\right) + \int_{K_{\vartheta\rho}}\xi^2\g_{-}(u,k)(x,-\rho^{2s})\,dx.
\end{align*}
For $i = 0, 1, \ldots$ we set
\begin{align*}
k_i = \frac{M}{2}+\frac{M}{2^{i+1}}, \quad \rho_i = \rho+\frac{\rho}{2^i}, \quad K_i = K_{\vartheta\rho_i}, \quad Q_i = K_i\times (-\rho_i^{2s},0],
\end{align*}
and
\begin{align*}
\widetilde{k}_i = \frac{k_i+k_{i+1}}{2}, \quad \widetilde{\rho}_i = \frac{\rho_i+\rho_{i+1}}{2}, \quad \widetilde{K}_i = K_{\vartheta\widetilde{\rho}_i}, \quad \widetilde{Q}_i = \widetilde{K}_i\times (-\widetilde{\rho}_i^{2s},0].
\end{align*}

For the cutoff function in the energy estimate, we work with $0\leq \xi \leq 1$ such that
\begin{align*}
|\nabla \xi| \leq C\frac{2^i}{\vartheta\rho}, \quad  |\de_t \xi| \leq C\frac{4^{is}}{\rho^{2s}} \quad \mbox{ and } \quad \xi = 1 \mbox{ on } \widetilde{Q}_{i}
\end{align*}
with the support of $\xi$ chosen to be away from $\de_P Q_i$ such that for any $S\in \text{supp}(\xi)$ and $P \in \bb{R}^n\setminus Q_i$ we have \cref{eq:SPP}. We set $A_i = \{u<k_i\} \cap Q_i$. The choice of the support of $\xi$ ensures that
\begin{align*}
\int_{K_i}\xi^2(u-k_i)_{-}^2(x,-\rho_i^{2s})\,dx = 0
\end{align*}
and
\begin{align*}
\underset{x\in \mbox{supp}(\xi(\cdot,t)); t\in (-\rho_{i}^{2s},0)}{\esup}\int_{\bb{R}^n\setminus K_i}\frac{(\phi(u)-\phi(k_i))_{-}(y,t)}{|x-y|^{n+2s}}\,dy \leq C2^{i(n+2s)}\underset{t\in (-\rho_{i}^{2s},0)}{\esup}\int_{\bb{R}^n\setminus K_i}\frac{(\phi(u)-\phi(k_i))_{-}(y,t)}{|y|^{n+2s}}\,dy.
\end{align*}

In particular, we can estimate the nonlocal term as follows:
\begin{align*}
&\int_{-\rho_{i}^{2s}}^0\int_{K_i}\xi^2(u-k_i)_{-}^m(x,t)\,dx\,dt\left(\underset{x\in \mbox{supp}(\xi(\cdot,t)); t\in (-\rho_{i}^{2s},0)}{\esup}\int_{\bb{R}^n\setminus K_i}\frac{(\phi(u)-\phi(k_i))_{-}(y,t)}{|x-y|^{n+2s}}\,dy\right)\\
&\leq C2^{i(n+2s)}\int_{-\rho_{i+1}^{2s}}^0\int_{K_i}(u-k_i)_{-}^m(x,t)\,dx\,dt\left(\underset{t\in (-\rho_{i}^{2s},0)}{\esup}\int_{\bb{R}^n\setminus K_i}\frac{(\phi(u)-\phi(k_i))_{-}(y,t)}{|y|^{n+2s}}\,dy\right)\\
&\leq C2^{i(n+2s)}|A_i|M^m\left(\frac{CM^m}{(\vartheta\rho)^{2s}}+\underset{t\in (-\rho_{i}^{2s},0)}{\esup}\int_{\bb{R}^n\setminus K_R}\frac{(\phi(u)-\phi(k_i))_{-}(y,t)}{|y|^{n+2s}}\,dy\right)\\
&\leq C2^{i(n+2s)}|A_i|M^m\left(\frac{M^{m}}{(\vartheta\rho)^{2s}}+\left(\frac{\vartheta\rho}{R}\right)^{2s}\frac{M^{m}}{(\vartheta\rho)^{2s}}+\left(\frac{\vartheta\rho}{R}\right)^{2s}\frac{\tl^m(u_{-},Q)}{(\vartheta\rho)^{2s}}\right),
\end{align*}
where we were able to replace the integral over $\bb{R}^n\setminus K_i$ with $\bb{R}^n\setminus K_R$ because of the choice of $M$ and the by the definition of $k_i$. Therefore assuming that
\begin{align*}
\left(\frac{\vartheta\rho}{R}\right)^{\frac{2s}{m}}\tl(u_{-},Q) \leq M
\end{align*}
and recalling that $\vartheta\rho\leq R$ by our choice of the cylinder, we get
\begin{align*}
\int_{-\vartheta\rho_{i+1}^{2s}}^0\int_{K_i}\xi^2(u-k_i)^m_{-}(x,t)\,dx\,dt\left(\underset{x\in \mbox{supp}(\xi(\cdot,t)); t\in (-\rho_{i}^{2s},0)}{\esup}\int_{\bb{R}^n\setminus K_i}\frac{(\phi(u)-\phi(k_i))_{-}(y,t)}{|x-y|^{n+2s}}\,dy\right) \leq C2^{i(n+2s)}|A_i|\frac{M^{2m}}{(\vartheta\rho)^{2s}}.
\end{align*}
Recalling that $2M \geq |u|$ locally and the properties of the test function, we estimate the remaining terms on the right hand side as follows:
\begin{align*}
\int_{-\rho_i^{2s}}^0\int_{K_i}\int_{K_i}\max\{(u-k_i)_{-}^{2m}(x,t),(u-k_i)_{-}^{2m}(y,t)\}\frac{|\xi(x,t)-\xi(y,t)|^2}{|x-y|^{n+2s}}\,dx\,dy\,dt \leq C\frac{4^{n}}{(\vartheta\rho)^{2s}}M^{2m}|A_i|
\end{align*}
and
\begin{align*}
\int_{-\rho_i^{2s}}^0\int_{K_i}(u-k_i)_{-}^{m+1}|\de_t \xi^2|\,dx\,dt \leq C\frac{4^{is}}{\rho^{2s}}M^{1+m}|A_i|,
\end{align*}

For the left land side of the energy estimate we estimate from below over the smaller region $\widetilde{K}_{\vartheta\rho}$ in space and $(-\widetilde{\rho}^{2s},0)$ in time where the test functions are identically unity. Next, we further estimate from below as follows:
\begin{align*}
\phi(k_i)-\phi(\widetilde{k}_i) = \int_{\widetilde{k}_i}^{k_i}\phi'(s)\,ds \geq mM^{m-1}(k_i-\widetilde{k}_i) = \frac{mM^m}{2^{i+3}},
\end{align*}
where we used that $M \geq k_i$ and
\begin{align*}
&\int_{-\widetilde{\rho}_i^{2s}}^0\int_{\widetilde{K}_i}\int_{\widetilde{K}_i}\min\{|u(x,t)|^{2(m-1)},|u(y,t)|^{2(m-1)}\}\frac{|(u-k)_{-}(x,t)-(u-k)_{-}(y,t)|^2}{|x-y|^{n+2s}}\,dx\,dy\,dt\\
&\geq (2M)^{2(m-1)}\int_{-\widetilde{\rho}_i^{2s}}^0\int_{\widetilde{K}_i}\int_{\widetilde{K}_i}\frac{|(u-k)_{-}(x,t)-(u-k)_{-}(y,t)|^2}{|x-y|^{n+2s}}\,dx\,dy\,dt
\end{align*}
using that $|u| \leq 2M$ and $m-1\leq 0$.

Putting together the above estimates yields
\begin{align}\label{eq:combine}
&\frac{M^m}{2^i}\esup_{t \in (-\widetilde{\rho}_i^{2s},0)}\int_{\widetilde{K}_i}(u-\widetilde{k}_i)_{-}\,dx 
+ M^{2(m-1)}\int_{-\widetilde{\rho}_i^{2s}}^0\int_{\widetilde{K}_i}\int_{\widetilde{K}_i}\frac{|(u-\widetilde{k}_n)_{-}(x,t)-(u-\widetilde{k}_n)_{-}(y,t)|^2}{|x-y|^{n+2s}}\,dx\,dy\,dt \\
&\leq C\left(\frac{4^{i}}{(\vartheta\rho)^{2s}}M^{2m}+\frac{4^{is}}{\rho^{2s}}M^{1+m}+\frac{2^{i(n+4)}}{(\vartheta\rho)^{2s}}M^{2m}\right)|A_i|.    
\end{align}
Let $0\leq \psi \leq 1$ be a cutoff function such that $\psi\equiv 0$ on $\partial_P\widetilde{Q}_i$ and $\psi = 1$ on $Q_{i+1}$ with $|\de \psi| \leq 2^i(\vartheta\rho)^{-1}$.  We apply H\"older inequality and Sobolev embedding \cref{eq:SP2} to get
\begin{align*}
&\left(\frac{M}{2^{i+3}}\right)^{2}|A_{i+1}|\leq \iint_{\widetilde{Q}_i}(u-\widetilde{k}_i)_{-}^2\psi^2 \,dx\,dt \\
&\leq\left(\iint_{\widetilde{Q}_i}((u-\widetilde{k}_i)\psi)_{-}^{2\frac{n+s}{n}} \,dx\,dt\right)^{\frac{n}{n+s}}|A_i|^{\frac{s}{n+s}}\\
&\leq \left(\int_{-\widetilde{\rho}_i^{2s}}^0\int_{\widetilde{K}_i}\int_{\widetilde{K}_i} \frac{|((u-\widetilde{k}_i)\psi)_{-}(x,t)-((u-\widetilde{k}_i)\psi)_{-}(y,t)|^2}{|x-y|^{n+2s}} \,dx\,dy\,dt \right)^{\frac{n}{n+s}}\left(\esup_{t\in I}\int_K |((u-\widetilde{k}_i)\psi)_{-}(x,t)|\,dx\right)^{\frac{2s}{n+s}}|A_i|^{\frac{s}{n+s}}.
\end{align*}
Using $|\nabla\psi|\leq 2^i(\vartheta\rho)^{-1}$ and the triangle inequality yields
\begin{align*}
&\int_{-\widetilde{\rho}_i^{2s}}^0\int_{\widetilde{K}_i}\int_{\widetilde{K}_i} \frac{|((u-\widetilde{k}_i)\psi)_{-}(x,t)-((u-\widetilde{k}_i)\psi)_{-}(y,t)|^2}{|x-y|^{n+2s}} \,dx\,dy\,dt \\
&\leq C\int_{-\widetilde{\rho}_i^{2s}}^0\int_{\widetilde{K}_i}\int_{\widetilde{K}_i} \frac{|(u-\widetilde{k}_i)_{-}(x,t)-(u-\widetilde{k}_i)_{-}(y,t)|^2}{|x-y|^{n+2s}} \,dx\,dy\,dt + C\frac{4^i}{(\vartheta\rho)^{2s}}\iint_{\widetilde{Q}_i}(u-\widetilde{k}_i)_{-}^2\,dx\,dt.
\end{align*}
From \cref{eq:combine} we get
\begin{align*}
\int_{-\widetilde{\rho}_i^{2s}}^0\int_{\widetilde{K}_i}\int_{\widetilde{K}_i} \frac{|(u-\widetilde{k}_i)_{-}(x,t)-(u-\widetilde{k}_i)_{-}(y,t)|^2}{|x-y|^{n+2s}} \,dx\,dy\,dt \leq C\frac{M^2}{M^{2m}}\left(\frac{4^{i}}{(\vartheta\rho)^{2s}}M^{2m}+\frac{4^{is}}{\rho^{2s}}M^{1+m}+\frac{2^{i(n+4)}}{(\vartheta\rho)^{2s}}M^{2m}\right)|A_i|.
\end{align*}
We also have
\begin{align*}
\frac{4^i}{(\vartheta\rho)^{2s}}\iint_{\widetilde{Q}_i}(u-\widetilde{k}_i)_{-}^2\,dx\,dt \leq C \frac{4^i}{(\vartheta\rho)^{2s}}M^2|A_i|.
\end{align*}
For the time term, using $0 \leq \psi \leq 1$, we get from above that
\begin{align*}
\esup_{t\in I}\int_K |((u-\widetilde{k}_i)\psi)_{-}(x,t)|\,dx \leq C\frac{2^i}{M^{m}}\left(\frac{4^{i}}{(\vartheta\rho)^{2s}}M^{2m}+\frac{4^{is}}{\rho^{2s}}M^{1+m}+\frac{2^{i(n+4)}}{(\vartheta\rho)^{2s}}M^{2m}\right)|A_i|.
\end{align*}
Therefore
\begin{align*}
\left(\frac{M}{2^{i+3}}\right)^{2}|A_{i+1}|\leq C\left(\frac{M^2}{M^{2m}}\right)^{\frac{n}{n+s}}\left(\frac{2^i}{M^{m}}\right)^{\frac{2s}{n+s}}\left(\frac{4^{is}}{\rho^{2s}}M^{1+m}+\frac{2^{i(n+4)}}{(\vartheta\rho)^{2s}}M^{2m}\right)^{\frac{n+2s}{n+s}}|A_i|^{1+\frac{2s}{n+s}},
\end{align*}
which implies
\begin{align*}
|A_{i+1}| \leq Cb^iM^{\frac{-2mn-2ms-2s}{n+s}}\left(\frac{M^{2m}}{(\vartheta\rho)^{2s}}\right)^{\frac{n+2s}{n+s}}\left(1+\vartheta^{2s}M^{1-m}\right)^{\frac{n+2s}{n+s}}|A_i|^{1+\frac{2s}{n+s}}.
\end{align*}	
Setting $Y_i = \frac{|A_i|}{|Q_i|}$, we get
\begin{align*}
Y_{i+1} \leq Cb^i\left(\frac{1}{\vartheta^{2s}M^{1-m}}\right)^{\frac{2s}{n+s}}\left(1+\vartheta^{2s}M^{1-m}\right)^{\frac{n+2s}{n+s}}Y_i^{1+\frac{2s}{n+s}}.
\end{align*}
Therefore, by Lemma \ref{lem:tech}, $Y_i \rightarrow 0$ provided
\begin{align*}
Y_0 \leq C \vartheta^{2s}M^{1-m} \left(1+\vartheta^{2s}M^{1-m}\right)^{-\frac{n+2s}{2s}}=:\nu'
\end{align*}
The conclusion follows.
\end{proof}

\begin{rem}\label{rem:deG-general-tail}
We can also assume
\begin{align*}
\left(\frac{\rho}{R}\right)^{\frac{2s}{m}}\tl(u_{\mp},Q) \leq CM\quad\text{in Lemma \ref{lem:deG-general}}\quad\text{and}\quad \left(\frac{\vartheta\rho}{R}\right)^{\frac{2s}{m}}\tl(u_{\mp},Q) \leq CM\quad\text{in Lemma \ref{lem:deG-general}},
\end{align*}
where $C\geq1$ is a constant that depends only on \textbf{data}. 
\end{rem}
\begin{rem}\label{rem:deG-ss}
With the choice of $\theta=M^{1-m}$ in Lemma \ref{lem:deG-general}, $\nu$ becomes a universal constant depending only on \textbf{data}. Also, with the choice of $\vartheta=M^{\frac{m-1}{2s}}$ in Lemma \ref{lem:deG-general'}, $\nu'$ becomes a universal constant depending only on \textbf{data}.
\end{rem}

\section{Analysis near the set $\{u \approx 0\}$ of singularity/degeneracy}\label{sec:close0}
Let us first assume $m\geq 1$. We work with a cylinder
\begin{align}\label{eq:Q}
(x_0,t_0)+Q_{\rho}(\theta) \subset \subset Q \subset \Omega_T.
\end{align}
Let $M>0$ be a number such that
\begin{align}\label{eq:bdd.M}
2M \geq \esup_{Q} |u| + \tl(u;Q).
\end{align}
Without loss of generality we assume that $(x_0,t_0) = (0,0)$. Recalling Remark \ref{rem:deG-ss} we set 
\begin{align*}
\theta = M^{1-m},
\end{align*}
and fix $\nu \in (0,1)$ to be the corresponding universal constant as obtained in Lemma \ref{lem:deG-general}. We will first work with the case when $u$ is close to $0$ and the far off effects are small i.e., we assume that the following conditions hold: 
\begin{align}\label{eq:tail.nu}
\left(\frac{\rho}{R}\right)^{\frac{2s}{m}}\tl(u_+,Q) \leq \pi^{\frac{1}{m}} M,\quad\quad |\{u \leq M\} \cap Q_{\rho}(\theta)| > \nu |Q_{\rho}(\theta)|
\end{align}
and
\begin{align}\label{eq:tail.nu2}
\left(\frac{\rho}{R}\right)^{\frac{2s}{m}}\tl(u_-,Q) \leq \pi^{\frac{1}{m}} M,\quad\quad |\{u \geq -M\} \cap Q_{\rho}(\theta)| > \nu |Q_{\rho}(\theta)|
\end{align}
with sufficiently small $\pi$ depending on \textbf{data} to be chosen later. As a first step, the smallness of far-off effects will follow by sufficiently zooming out (see \cref{prop:close-to-0}). However, when we iterate it down to smaller scales we must carefully add up the nonlocal effects - this is done in the next section.

When \cref{eq:tail.nu} holds, we claim that 
\begin{align*}
u \leq 2M(1-\eta) \mbox{ a.e. in } Q_{\frac{\rho}{2}}(\nu\theta)
\end{align*}
whereas if \cref{eq:tail.nu2} holds, then we will get
\begin{align*}
u \geq - 2M(1-\eta) \mbox{ a.e. in } Q_{\frac{\rho}{2}}(\nu\theta)
\end{align*}
so that 
\begin{align*}
\esup_{Q_{\frac{\rho}{2}}(\nu\theta)} |u| \leq 2M(1-\eta),
\end{align*}
where $\eta \in (0,1)$ is a universal constant.

We note that the equation \cref{eq:main-eq} is invariant under a sign-change and that the condition \cref{eq:tail.nu2} is exactly the condition \cref{eq:tail.nu} but for $-u$. Therefore, there is no loss of generality in working with \cref{eq:tail.nu}. Note that by \cref{eq:bdd.M} we have 
\begin{align*}
|u| \leq 2M \mbox{ in } Q_{\rho}(\theta).
\end{align*}
Moreover, we will study $u$ in the set
\begin{align}\label{eq:2M}
\{M \leq |u| \leq 2M\} \cap Q_{\rho}(\theta),
\end{align}
In summary, we work under the condition
\begin{equation}\label{eq:close-to-0}
|\{u \leq M\} \cap Q_{\rho}(\theta)| > \nu |Q_{\rho}(\theta)| 
\end{equation}
in the set
\begin{equation}\label{eq:u-is-M}
\{M \leq |u| \leq 2M\} \cap Q_{\rho}(\theta)
\end{equation}

Now let us assume $0<m\leq 1$. We work with a cylinder
\begin{align}\label{eq:Q'}
(x_0,t_0)+\mathfrak{Q}_{\rho}(\vartheta) \subset \subset Q \subset \Omega_T.
\end{align}
Let $M>0$ be a number such that
\begin{align}\label{eq:bdd.M'}
2M \geq \esup_{Q} |u| + \tl^m(u;Q).
\end{align}
Without loss of generality we assume that $(x_0,t_0) = (0,0)$. Recalling Remark \ref{rem:deG-ss} we set 
\begin{align*}
\vartheta = M^{\frac{m-1}{2s}},
\end{align*}
and fix $\nu' \in (0,1)$ to be the corresponding universal constant as obtained in Lemma \ref{lem:deG-general'}. We will first work with the case when $u$ is close to $0$ and the far off effects are small i.e., we assume that the following conditions hold: 
\begin{align}\label{eq:tail.nu'}
\left(\frac{\vartheta\nu'^{-\frac{1}{2s}}\rho}{R}\right)^{\frac{2s}{m}}\tl(u_+,Q) \leq \pi'^{\frac{1}{m}} M,\quad\quad |\{u \leq M\} \cap \mathfrak{Q}_{\rho}(\vartheta)| > \nu' |\mathfrak{Q}_{\rho}(\vartheta)|
\end{align}
and
\begin{align}\label{eq:tail.nu2'}
\left(\frac{\vartheta\nu'^{-\frac{1}{2s}}\rho}{R}\right)^{\frac{2s}{m}}\tl(u_-,Q) \leq \pi'^{\frac{1}{m}} M,\quad\quad |\{u \geq -M\} \cap \mathfrak{Q}_{\rho}(\vartheta)| > \nu' |\mathfrak{Q}_{\rho}(\vartheta)|
\end{align}
with sufficiently small $\pi'$ depending on \textbf{data}. As a first step, the smallness of far-off effects will follow by sufficiently zooming out. (see \cref{prop:close-to-0}). However, when we iterate it down to smaller scales we must carefully add up the nonlocal effects - this is done in the next section.

When \cref{eq:tail.nu'} holds, we claim that 
\begin{align*}
u \leq 2M(1-\eta') \mbox{ a.e. in } \mathfrak{Q}_{\frac{\rho}{2}}(\nu'^{-\frac{1}{2s}}\vartheta)
\end{align*}
whereas if \cref{eq:tail.nu2'} holds, then we will get
\begin{align*}
u \geq - 2M(1-\eta') \mbox{ a.e. in } \mathfrak{Q}_{\frac{\rho}{2}}(\nu'^{-\frac{1}{2s}}\vartheta)
\end{align*}
so that 
\begin{align*}
\esup_{\mathfrak{Q}_{\frac{\rho}{2}}(\nu'^{-\frac{1}{2s}}\vartheta)} |u| \leq 2M(1-\eta'),
\end{align*}
where $\eta' \in (0,1)$ is a universal constant.

We note that the equation \cref{eq:main-eq} is invariant under a sign-change and that the condition \cref{eq:tail.nu2'} is exactly the condition \cref{eq:tail.nu'} but for $-u$. Therefore, there is no loss of generality in working with \cref{eq:tail.nu'}. Note that by \cref{eq:bdd.M'} we have 
\begin{align*}
|u| \leq 2M \mbox{ in } \mathfrak{Q}_{\rho}(\vartheta).
\end{align*}
Moreover, we will study $u$ in the set
\begin{align}\label{eq:2M'}
\{M \leq |u| \leq 2M\} \cap \mathfrak{Q}_{\rho}(\vartheta),
\end{align}
In summary, we work under the condition
\begin{equation}\label{eq:close-to-0'}
|\{u \leq M\} \cap \mathfrak{Q}_{\rho}(\vartheta)| > \nu |\mathfrak{Q}_{\rho}(\vartheta)| 
\end{equation}
in the set
\begin{equation}\label{eq:u-is-M'}
\{M \leq |u| \leq 2M\} \cap \mathfrak{Q}_{\rho}(\vartheta)
\end{equation}

We prove the following lemma. Recall that $\theta=M^{1-m}$.

\begin{lem}\label{lem:time-slice}
Suppose that \cref{eq:close-to-0} holds. Then there is a time $\tau \in \left(-\theta\rho^{2s},-\frac{\nu}{2}\theta\rho^{2s}\right]$ such that
\begin{equation}\label{eq:time-slice}
\left|\left\{u(\cdot,\tau) \leq M \right\} \cap K_{\rho}\right| > \frac{\nu}{2}\left|K_{\rho}\right|.
\end{equation}
\end{lem}
\begin{proof}
If not, then
\begin{align*}
|\{u \leq M\} \cap Q_{\rho}(\theta)| &\leq \int_{-\theta\rho^{2s}}^{-\theta\rho^{2s}\nu/2} \left|\left\{u(\cdot,t) \leq M \right\} \cap K_{\rho}\right| \,dt+ \frac{\theta\rho^{2s}\nu}2\left|K_{\rho}\right|\\
&\leq\theta\frac{\nu}{2}\left(1-\frac{\nu}2\right)\rho^{2s}\left|K_{\rho}\right|+\frac{\theta\nu}2\rho^{2s}\left|K_{\rho}\right|<\nu |Q_{\rho}(\theta)|
\end{align*}
and
\begin{align*}
\frac{1}{2}\nu\left|K_{\rho}\right|\theta\rho^{2s} = \frac{1}{2}\nu |Q_{\rho}(\theta)|,
\end{align*}
which contradicts \cref{eq:close-to-0}.
\end{proof}
Similarly, one can prove the following lemma. Recall that $\vartheta=M^{\frac{m-1}{2s}}$.
\begin{lem}\label{lem:time-slice'}
Suppose that 
\begin{align}\label{eq:nu'}
|\{u\leq M\}\cap\mathfrak{Q}_{\rho}(\vartheta)|>\frac{\nu'}{2}|\mathfrak{Q}_{\rho}(\vartheta)|
\end{align}
holds. Then there is a time $\tau \in \left(-\rho^{2s},-\frac{\nu'^{1+\frac{n}{2s}}}{2}\rho^{2s}\right]$ such that
\begin{equation}\label{eq:time-slice'}
\left|\left\{u(\cdot,\tau) \leq M \right\} \cap K_{\vartheta\rho\nu'^{-\frac{1}{2s}}}\right| > \frac{\nu'^{1+\frac{n}{2s}}}{2}\left|K_{\vartheta\rho\nu'^{-\frac{1}{2s}}}\right|.
\end{equation}
\end{lem}
\begin{proof}
If not, then
\begin{align*}
|\{u \leq M\} \cap Q_{\rho}(\vartheta)| &\leq \int_{-\rho^{2s}}^{-\rho^{2s}\nu'^{1+\frac{n}{2s}}/2} \left|\left\{u(\cdot,t) \leq M \right\} \cap K_{\vartheta\rho\nu'^{-\tfrac{1}{2s}}}\right| \,d\tau+ \frac{\rho^{2s}\nu'^{1+\frac{n}{2s}}}2\left|K_{\vartheta\rho\nu'^{-\frac1{2s}}}\right|\\
&\leq \frac{\nu'^{1+\frac{n}{2s}}}2\left(1-\frac{\nu'^{1+\frac{n}{2s}}}2\right)\rho^{2s}\left|K_{\vartheta\rho\nu'^{-\frac1{2s}}}\right|+\frac{\nu'^{1+\frac{n}{2s}}}2\rho^{2s}\left|K_{\vartheta\rho\nu'^{-\frac1{2s}}}\right|<\nu' |Q_{\rho}(\vartheta)|
\end{align*}
and
\begin{align*}
\frac{1}{2}\nu'^{1+\frac{n}{2s}}\left|K_{\vartheta\rho\nu'^{-\tfrac{1}{2s}}}\right|\rho^{2s} = \frac{1}{2}\nu' |Q_{\rho}(\vartheta)|,
\end{align*}
which contradicts \cref{eq:close-to-0'}.
\end{proof}

Now we prove several lemmas. The following is a time propagation lemma.

\begin{lem}\label{prop:time-exp}
Let $m\geq 1$. Suppose that \cref{eq:time-slice} holds. Then there is a $\delta \in (0,1)$ and $\epsilon \in (0,\tfrac{1}{4})$ which depend only on $\nu$ and the \textbf{data} such that
\begin{equation}\label{eq:time-exp}
\left|\{u(\cdot,t) \leq 2M - \epsilon M\} \cap K_{\rho}\right| \geq \frac{\nu}{4}\left|K_{\rho}\right| \mbox{ for all times } t \in (\tau,\tau+\delta\theta\rho^{2s}],
\end{equation}
provided 
\begin{align}\label{eq:tail}
\left(\frac{\rho}{R}\right)^{\frac{2s}{m}}\tl(u_{+},Q) \leq M 
\end{align}
holds. 
\end{lem}
\begin{proof}
For ease of notation, let us denote $K_{\rho}$ by $K$ simply and let $Q' = K \times (\tau,\tau+\delta\theta\rho^{2s}]$. By $rK$ we mean $K$ but with the radius linearly scaled by $r$. Note that since $\delta\in(0,1)$ and \cref{eq:Q},
\begin{align*}
Q' \subset Q.
\end{align*}
We use Lemma \ref{lem:ee} along with a time independent cutoff function $\xi = \xi(x)$ in $K$ such that $\xi = 1$ on $(1-\sigma)K$, with support away from $\de K$ such that
\begin{align}\label{eq:grad.xi}
|\nabla \xi| \leq \frac{10}{\sigma\rho}.
\end{align}
Then for all times $t \in (\tau,\tau+\delta\theta\rho^{2s}]$ we have 
\begin{align*}
\int_K (u-M)_+^2(x,t)\xi^2(x) \,dx \leq E+F+G,
\end{align*}
where 
\begin{align*}
E = \int_{K}(u-M)_+^2(x,\tau)\xi^2(x) \,dx,
\end{align*}
\begin{align*}
F =
C\displaystyle\int_Q\displaystyle\int_K\max\{|u|^{m-1}(x,t),|u|^{m-1}(y,t)\}\max\{(u-M)_{+}^{2}(x,t),(u-M)_{+}^{2}(y,t)\}\frac{|\xi(x)-\xi(y)|^2}{|x-y|^{n+2s}}\,dx\,dy\,dt,
\end{align*}
and
\begin{align*}
G = C\int_Q\xi^2(u-M)_{+}(x,t)\,dx\left(\underset{x\in \mbox{supp}(\xi(\cdot,t)); t\in (\tau,\tau+\delta\theta\rho^{2s})}{\esup}\int_{\bb{R}^n\setminus K}\frac{(\phi(u)-\phi(M))_{+}(y,t)}{|x-y|^{n+2s}}\,dy\right)\,dt.
\end{align*}
Choosing the support of $\xi$ appropriately as in the proof of Lemma \ref{lem:deG-general} by
\begin{align*}
&\underset{x\in \mbox{supp}(\xi(\cdot,t)); t\in (\tau,\tau+\delta \theta\rho^{2s})}{\esup}\int_{\bb{R}^n\setminus K}\frac{(\phi(u)-\phi(M))_{+}(y,t)}{|x-y|^{n+2s}}\,dy\\
&\quad\leq \frac{1}{\sigma^{n+2s}}\esup_{t \in (T_1,T_2]} \int_{K_R\setminus K} \frac{\phi(M)(y,t)}{|y|^{n+2s}}\,dy + \frac{1}{\sigma^{n+2s}}\esup_{t \in (T_1,T_2]}\int_{\bb{R}^n\setminus K_R} \frac{\phi(u)_+(y,t)}{|y|^{n+2s}}\,dy\\
&\quad\leq \frac{C}{\sigma^{n+2s}}\left(\frac{M^m}{\rho^{2s}}+\frac{\tl^m(u_+,Q)}{R^{2s}}\right).
\end{align*}
Therefore enforcing \cref{eq:tail}, and recalling that $\rho\leq R$ by our choice of the cylinder, we get
\begin{align*}
\underset{x\in \mbox{supp}(\xi(\cdot,t)); t\in (\tau,\tau+\delta\theta\rho^{2s})}{\esup}\int_{\bb{R}^n\setminus K}\frac{(\phi(u)-\phi(M))_{+}(y,t)}{|x-y|^{n+2s}}\,dy \leq \frac{C}{\sigma^{n+2s}}\frac{M^m}{\rho^{2s}}.
\end{align*}
Hence we obtain
\begin{align*}
G \leq C|K|\delta M^{1-m}\rho^{2s}M\frac{1}{\sigma^{n+2s}}\frac{M^m}{\rho^{2s}} = C\frac{\delta M^2|K|}{\sigma^{n+2s}}.
\end{align*}
For $F$ we get
\begin{align*}
F\leq CM^{m-1}M^2\dfrac{\rho^{2-2s}}{(\sigma\rho)^2}|K|\delta M^{1-m}\rho^{2s} \leq C\frac{\delta M^2}{\sigma^{2}}|K|.
\end{align*}
For $E$ we make use of the hypothesis \cref{eq:time-slice} to get
\begin{align*}
E\leq M^2\left(1-\frac{1}{2}\nu\right)|K|.
\end{align*}
We finally estimate from below
\begin{align*}
\int_K (u-M)_+^2\xi^2(x) \,dx \geq \int_{(1-\sigma)K \cap \{u \geq 2M-\epsilon M\}} (u-M)_+^2\xi^2(x) \,dx \geq (1-\epsilon)^2M^2|\{u(\cdot,t) \geq 2M-\epsilon M\}\cap(1-\sigma)K|. 
\end{align*}
We also note that 
\begin{align*}
|K \setminus (1-\sigma)K| \leq n\sigma|K|.
\end{align*}
Putting together the estimates above we find
\begin{align*}
|\{u(\cdot,t) \geq 2M-\epsilon M\}\cap K| \leq \frac{1}{(1-\epsilon)^2}\left(\left(1-\frac{1}{2}\nu\right)+C\frac{\delta}{\sigma^{2}}+\frac{C\delta}{\sigma^{n+2s}}\right)|K| +n\sigma|K|.
\end{align*}
This allows us to choose the various parameters quantitatively. Indeed, we can assume $\epsilon$ is so small that $(1-\epsilon)^2\geq\frac{1}{2}$, i.e., $\epsilon\leq\frac{1}{4}$ is enough. Then choose $\sigma$ and $\delta$ with
\begin{align*}
\sigma=\frac{\nu}{16n},\quad C\left(\frac{1}{\sigma^{2}}+\frac{1}{\sigma^{n+2s}}\right)\delta\leq\frac{\nu}{32}.
\end{align*}
Then the above estimate is simplified as 
\begin{align*}
|\{u(\cdot,t) \geq 2M-\epsilon M\}\cap K| \leq \left[\dfrac{1-\frac{1}{2}\nu}{(1-\epsilon)^2}+\frac{\nu}{8}\right]|K|.
\end{align*}
To conclude, we may assume $\epsilon$ such that
\begin{align*}
\dfrac{1-\frac{1}{2}\nu}{(1-\epsilon)^2}+\frac{\nu}{8}<1-\frac{\nu}{4},\quad\text{i.e.,}\quad\epsilon\leq\min\left\{\frac{\nu}{8},\frac{1}{4}\right\}.
\end{align*}
We complete the proof.
\end{proof}

By a simple modification of the above argument, we also have the following lemma for $0<m\leq 1$. 
\begin{lem}\label{prop:time-exp'}
Let $0<m\leq 1$. Suppose that \cref{eq:time-slice'} holds. Then there is a $\delta' \in (0,1)$ and $\epsilon' \in (0,\tfrac{1}{4})$ which depend only on $\nu'$ and the \textbf{data} such that
\begin{equation}\label{eq:time-exp'}
\left|\{u(\cdot,t) \leq 2M - \epsilon' M\} \cap K_{\vartheta\rho\nu'^{-\frac{1}{2s}}}\right| \geq \frac{\nu'^{1+\frac{n}{2s}}}{4}\left|K_{\vartheta\rho\nu'^{-\frac{1}{2s}}}\right| \mbox{ for all times } t \in (\tau,\tau+\delta'\rho^{2s}],
\end{equation}
provided 
\begin{align}\label{eq:tail'}
\left(\frac{\vartheta\nu'^{-\frac{1}{2s}}\rho}{R}\right)^{\frac{2s}{m}}\tl(u_{+},Q) \leq M 
\end{align}
holds. 
\end{lem}
\begin{proof}
For ease of notation, let us denote $K_{\vartheta\nu'^{-\frac{1}{2s}}\rho}$ by $K$ simply and let $Q' = K \times (\tau,\tau+\delta'\rho^{2s}]$. By $rK$ we mean $K$ but with the radius linearly scaled by $r$. Note that since $\delta\in(0,1)$ and \cref{eq:Q'},
\begin{align*}
Q' \subset Q.
\end{align*}
We use Lemma \ref{lem:ee.m<1} along with a time independent cutoff function $\xi = \xi(x)$ in $K$ such that $\xi = 1$ on $(1-\sigma)K$, with support away from $\de K$ such that
\begin{align}\label{eq:grad.xi'}
|\nabla \xi| \leq \frac{10\nu'^{\frac{1}{2s}}}{\sigma\vartheta\rho}.
\end{align}
Then arguing similarly to the proof of Lemma \ref{prop:time-exp} using \cref{eq:tail'} instead of \cref{eq:tail} yields the conclusion. The tail condition changes because we now scale in space instead of time. 
\end{proof}

\begin{lem}\label{prop:shrink}
Let $m\geq 1$. Suppose that \cref{eq:time-slice} holds. Then for any positive integer $j \geq 1$ we have 
\begin{align*}
\left|\left\{u \geq 2M - \frac{\epsilon M}{2^{j+1}} \right\} \cap Q'\right| \leq \frac{C}{\delta \nu} \frac{1}{2^j}|Q'| \quad \mbox{ for } \quad Q' =   K_{\rho} \times (\tau,\tau+\delta\theta\rho^{2s}],
\end{align*}
where $C>0$ is a universal constant provided
\begin{align}\label{eq:tail.as}
\left(\frac{\rho}{R}\right)^{\frac{2s}{m}}\tl(u_{+},Q) \leq \left(\frac{\epsilon}{2^j}\right)^{\frac{1}{m}}M
\end{align}
holds. 
\end{lem}
\begin{proof}
For ease of notation, let us denote $K_{2\rho}$ by just $K$ and let $Q'' = K \times (\tau,\tau+\delta\theta\rho^{2s}] = K \times I$. Note that we can assume
\begin{align*}
Q'' \subset Q. 
\end{align*}
We choose a test function $\xi = \xi(x)$ independent of time and $\xi\equiv 1$ on $\tfrac{1}{2}K:=K_{\rho}$ with its support away from $\de K$ such that 
\begin{align*}
|\nabla \xi| \leq \frac{10}{\rho}.
\end{align*}
For a fixed $j \geq 1$, we set 
\begin{align*}
l = 2M - \frac{\epsilon M}{2^{j}}>0.
\end{align*}
We note that on the set $\{u>l\}$ we have $u \geq M/2$. We now apply Lemma \ref{lem:ee} to get
\begin{align*}
\int_I\int_{\frac{1}{2}K}(u-l)_{+}(x,t)\,dx \left(\int_K \frac{(\phi(u)-\phi(l))_{-}(y,t)}{|x-y|^{n+2s}}\,dy\right)\,dt \leq \mathfrak{P} + \mathfrak{Q} + \mathfrak{R},
\end{align*}
where 
\begin{align*}
\mathfrak{P} = \int_{K}(u-l)_+^2(x,\tau)\xi^2(x) \,dx,
\end{align*}
\begin{align*}
\mathfrak{Q} =
C\displaystyle\int_{Q''}\displaystyle\int_K\max\{|u|^{m-1}(x,t),|u|^{m-1}(y,t)\}\max\{(u-l)_{+}^{2}(x,t),(u-l)_{+}^{2}(y,t)\}\frac{|\xi(x)-\xi(y)|^2}{|x-y|^{n+2s}}\,dx\,dy\,dt,
\end{align*}
and
\begin{align*}
\mathfrak{R} = C\int_{Q''}\xi^2(u-l)_{+}(x,t)\,dx\left(\underset{x\in \mbox{supp}(\xi(\cdot,t)); t\in (\tau,\tau+\delta\theta\rho^{2s})}{\esup}\int_{\bb{R}^n\setminus K}\frac{(\phi(u)-\phi(l))_{+}(y,t)}{|x-y|^{n+2s}}\,dy\right)\,dt.
\end{align*}
Choosing the support of $\xi$ appropriately we can estimate the tail term as in the proof of Lemma \ref{lem:deG-general} by
\begin{align*}
&\underset{x\in \mbox{supp}(\xi(\cdot,t)); t\in (\tau,\tau+\delta\theta\rho^{2s})}{\esup}\int_{\bb{R}^n\setminus K}\frac{(\phi(u)-\phi(l))_{+}(y,t)}{|x-y|^{n+2s}}\,dy\\
&\leq \esup_{t \in (T_1,T_2]} \int_{K_R\setminus K} \frac{(\phi(u)-\phi(l))_{+}(y,t)}{|y|^{n+2s}}\,dy+\esup_{t \in (T_1,T_2]}\int_{\bb{R}^n\setminus K_R} \frac{(\phi(u)-\phi(l))_{+}(y,t)}{|y|^{n+2s}}\,dy.
\end{align*} 
Recalling that $|u| \leq 2M$ along with the mean value theorem shows that
\begin{align*}
\esup_{t \in (T_1,T_2]} \int_{K_R\setminus K} \frac{(\phi(u)-\phi(l))_{+}(y,t)}{|y|^{n+2s}}\,dy \leq  \esup_{t \in (T_1,T_2]} \int_{K_R\setminus K} \frac{(\phi(u)-\phi(u - \tfrac{\epsilon M}{2^j}))_{+}(y,t)}{|y|^{n+2s}}\,dy \leq C\frac{\epsilon M}{2^j}\frac{M^{m-1}}{\rho^{2s}} \leq C \frac{\epsilon M^m}{2^j\rho^{2s}}.
\end{align*}
Next, we note that $\phi(l) \geq 0$ globally, and so we may estimate 
\begin{align*}
\int_{\bb{R}^n\setminus K_R} \frac{(\phi(u)-\phi(l))_{+}(y,t)}{|y|^{n+2s}}\,dy \leq \int_{\bb{R}^n\setminus K_R} \frac{\phi(u)_{+}(y,t)}{|y|^{n+2s}}\,dy \leq  C\left(\frac{\rho}{R}\right)^{2s}\frac{\tl^m(u_{+},Q)}{\rho^{2s}}.
\end{align*}
Therefore enforcing \cref{eq:tail.as}, we get
\begin{align*}
\underset{x\in \mbox{supp}(\xi(\cdot,t)); t\in (\tau,\tau+\delta\theta\rho^{2s})}{\esup}\int_{\bb{R}^n\setminus K}\frac{(\phi(u)-\phi(l))_{+}(y,t)}{|x-y|^{n+2s}}\,dy \leq  C \frac{\epsilon M^m}{2^j\rho^{2s}}.
\end{align*}
Hence we obtain
\begin{align*}
\mathfrak{R}\leq C|Q''|\frac{\epsilon M}{2^j}\frac{\epsilon M^m}{2^j\rho^{2s}} \leq C\left(\frac{\epsilon M}{2^j}\right)^2\frac{M^{m-1}}{\rho^{2s}}|Q'|.
\end{align*}
For $ \mathfrak{Q}$, recalling $u \geq M/2$ we have,
\begin{align*}
\mathfrak{Q} \leq CM^{m-1}\left(\frac{\epsilon M}{2^{j}}\right)^2\frac{1}{\rho^{2s}}|Q''| \leq C\left(\frac{\epsilon M}{2^j}\right)^2\frac{M^{m-1}}{\rho^{2s}}|Q'|.
\end{align*}
For $ \mathfrak{P}$  we have
\begin{align*}
\mathfrak{P} \leq C\left(\frac{\epsilon M}{2^{j}}\right)^2|K| \leq C\left(\frac{\epsilon M}{2^j}\right)^2\frac{M^{m-1}}{\delta\rho^{2s}}|Q'|.
\end{align*}
We have shown that 
\begin{align*}
\int_I\int_{\frac{1}{2}K}(u-l)_{+}(x,t)\,dx \left(\int_K \frac{(\phi(u)-\phi(l))_{-}(y,t)}{|x-y|^{n+2s}}\,dy\right)\,dt \leq C\left(\frac{\epsilon M}{2^j}\right)^2\frac{M^{m-1}}{\delta\rho^{2s}}|Q'|.
\end{align*}
With \cref{eq:tail.as}, by Lemma \ref{prop:time-exp} and the hypothesis \cref{eq:time-slice} we have \cref{eq:time-exp}. In Lemma \ref{lem:iso} we now put
\begin{align*}
k = 2M - \epsilon M \quad \mbox{,} \quad l = 2M - \frac{\epsilon M}{2^j} \quad \mbox{and} \quad q = 2M - \frac{\epsilon M}{2^{j+1}}.
\end{align*}
For each time slice $t \in (\tau,\tau+\delta\theta\rho^{2s}]$ we have 
\begin{align*}
\left|\{u(\cdot,t) \leq k\} \cap K_{\rho}\right| \geq \frac{\nu}{4}\left|K_{\rho}\right|
\end{align*}
by \cref{eq:time-exp}. Integrating in time, Lemma \ref{lem:iso} says that
\begin{align*}
\frac{(\epsilon M)^2}{2^j} M^{m-1}\frac{\nu}{4}|K||\{u \geq m\} \cap Q'| \leq C\left(\frac{\epsilon M}{2^j}\right)^2\frac{M^{m-1}}{\delta\rho^{2s}}\rho^{n+2s}|Q'|.
\end{align*}
Simplifying the above expression yields 
\begin{align*}
|\{u \geq q\} \cap Q'| \leq \frac{C}{\delta \nu} \frac{1}{2^j}|Q'|.
\end{align*}
The conclusion follows.
\end{proof}

\begin{lem}\label{prop:shrink'}
Let $0<m\leq 1$. Suppose that \cref{eq:time-slice'} holds. Then for any positive integer $j \geq 1$ we have 
\begin{align*}
\left|\left\{u \geq 2M - \frac{\epsilon' M}{2^{j+1}} \right\} \cap Q'\right| \leq \frac{C}{\delta' \nu'^2} \frac{1}{2^j}|Q'| \quad \mbox{ for } \quad Q' =   K_{\vartheta\rho\nu'^{-\frac{1}{2s}}} \times (\tau,\tau+\delta'\rho^{2s}],
\end{align*}
where $C>0$ is a universal constant provided
\begin{align}\label{eq:tail.as'}
\left(\frac{\vartheta\nu'^{-\frac{1}{2s}}\rho}{R}\right)^{\frac{2s}{m}}\tl(u_{+},Q) \leq \left(\frac{\epsilon'}{2^j}\right)^{\frac{1}{m}}M
\end{align}
holds. 
\end{lem}
\begin{proof}
For ease of notation, let us denote $K_{2\vartheta\rho\nu'^{-\frac{1}{2s}}}$ by just $K$ and let $Q'' = K \times (\tau,\tau+\delta'\rho^{2s}] = K \times I$. Note that we can assume
\begin{align*}
Q'' \subset Q. 
\end{align*}
We choose a test function $\xi = \xi(x)$ independent of time and $\xi\equiv 1$ on $\tfrac{1}{2}K$ with its support away from $\de K$ such that 
\begin{align*}
|\nabla \xi| \leq \frac{10\nu'^{\frac{1}{2s}}}{\vartheta\rho}.
\end{align*}
For a fixed $j \geq 1$, we set 
\begin{align*}
l = 2M - \frac{\epsilon' M}{2^{j}}>0.
\end{align*}
We note that on the set $\{u>l\}$ we have $u \geq M/2$. We now apply Lemma \ref{lem:ee.m<1} and arguing similarly to the proof of Lemma \ref{prop:shrink} using \cref{eq:tail.as'} instead of \cref{eq:tail.as} yields the conclusion. Note that we get $\nu'^2$ instead of $\nu'$ again due to space scaling. 
\end{proof}
\begin{lem}\label{prop:critical-mass}
Let $m\geq 1$ and $\mathfrak{f} \in (0,1)$. Then we can find a $\mu > 0$ independent of $\mathfrak{f}$, but depends on $\delta$ and the \textbf{data} such that if 
\begin{align*}
|\{u \geq 2M-\mathfrak{f}M\} \cap Q'| \leq \mu |Q'| \quad \mbox{ for } \quad Q' =   K_{\rho} \times (\tau,\tau+\delta \theta\rho^{2s}]
\end{align*}
and 
\begin{align}\label{eq:tail.as2}
\left(\frac{\rho}{R}\right)^{\frac{2s}{m}}\tl(u_{+},Q) \leq \mathfrak{f}^{\frac{1}{m}}M
\end{align}
hold, then
\begin{align*}
u \leq 2M - \frac{1}{2}\mathfrak{f}M \quad \mbox{ a.e. in } K_{\frac{\rho}{2}} \times (\tau+\tfrac{3}{4}\delta \theta\rho^{2s},\tau+\delta \theta\rho^{2s}].
\end{align*}
\end{lem}

\begin{proof}
Without loss of generality we assume that $\tau+\delta \theta\rho^{2s} = 0$.  For $i = 0, 1, \ldots$ we set
\begin{align*}
k_i = 2M-\frac{\mathfrak{f}M}{2}-\frac{\mathfrak{f}M}{2^{i+1}}, \quad \rho_i = \frac{\rho}{2}+\frac{\rho}{2^{i+1}}, \quad K_i = K_{\rho_i}, \quad Q_i = K_i \times I_i = K_i\times (-\delta \theta\rho_i^{2s},0].
\end{align*}
We note that on the set $\{u>k_i\}$ we have
\begin{align*}
M < u \leq 2M.
\end{align*}
We choose the cutoff function $\xi$ such that $0\leq \xi\leq 1$, $\xi\equiv 1$ on $Q_{i+1}$ with its support away from $\partial_P Q_i$ with 
\begin{align*}
|\nabla \xi|\leq C\frac{2^i}{\rho}\quad\text{and}\quad |\partial_t\xi|\leq C\frac{4^{is}}{\delta \theta\rho^{2s}}.
\end{align*}
Therefore, we may apply Lemma \ref{lem:ee} to get
\begin{align}\label{eq:cacc2}
\begin{split}
&\esup_{t \in I_{i}}\int_{K_{i}}\xi^2(u-k_i)^2_{+}\,dx\\
&\quad+\int_{I_i}\int_{K_i}\int_{K_i} \min\{\xi^2(x,t),\xi^2(y,t)\}\frac{((\phi(u)-\phi(k_i))_{+}(x,t)-(\phi(u)-\phi(k_i))_{+}(y,t))((u-k_i)_{+}(x,t)-(u-k_i)_{+}(y,t))}{|x-y|^{n+2s}}\,dx\,dy\,dt\\
&\quad\leq \mathfrak{P} + \mathfrak{Q} + \mathfrak{R},
\end{split}
\end{align}
where
\begin{align*}
\mathfrak{P} =
C\displaystyle\iint_{Q_{i}}\int_{K_i}\max\{|u|^{m-1}(x,t),|u|^{m-1}(y,t)\}\max\{(u-k_i)_{+}^{2}(x,t),(u-k_i)_{+}^{2}(y,t)\}\frac{|\xi(x,t)-\xi(y,t)|^2}{|x-y|^{n+2s}}\,dx\,dy\,dt,
\end{align*}
\begin{align*}
\mathfrak{Q} = C\int_{I_{i}}\int_{K_i}(u-k_i)_{+}^{2}|\de_t \xi^2|\,dx\,dt,
\end{align*}		
and
\begin{align*}
\mathfrak{R} =   C\int_{I_{i}}\int_{K_i}\xi^2(u-k_i)_{+}(x,t)\,dx\left(\underset{x\in \mbox{supp}(\xi(\cdot,t)); t\in I_{i}}{\esup}\int_{\bb{R}^n\setminus {K_i}}\frac{(\phi(u)-\phi(k_i))_{+}(y,t)}{|x-y|^{n+2s}}\,dy\right)\,dt.
\end{align*}
We choose the support of $\xi$ appropriately as in the proof of Lemma \ref{lem:deG-general} and estimate the terms similarly. We get
\begin{align*}
&\underset{x\in \mbox{supp}(\xi(\cdot,t)); t\in {I_i}}{\esup}\int_{\bb{R}^n\setminus {K_i}}\frac{(\phi(u)-\phi(k_i))_{+}(y,t)}{|x-y|^{n+2s}}\,dy\\
&\quad\leq 2^{i(n+2s)}\esup_{t \in (T_1,T_2]} \int_{K_R\setminus K_i} \frac{(\phi(u)-\phi(l))_{+}(y,t)}{|y|^{n+2s}}\,dy+2^{i(n+2s)}\esup_{t \in (T_1,T_2]}\int_{\bb{R}^n\setminus K_R} \frac{(\phi(u)-\phi(l))_{+}(y,t)}{|y|^{n+2s}}\,dy.
\end{align*}
Recalling that $|u| \leq 2M$ and the mean value theorem shows that
\begin{align*}
\esup_{t \in (T_1,T_2]} \int_{K_R\setminus K_i} \frac{(\phi(u)-\phi(k_i))_{+}(y,t)}{|y|^{n+2s}}\,dy\leq  \esup_{t \in (T_1,T_2]} \int_{K_R\setminus K_i} \frac{(\phi(u)-\phi(u - \tfrac{\mathfrak{f}M}{2}-\tfrac{\mathfrak{f}M}{2^{i+1}}))_{+}(y,t)}{|y|^{n+2s}}\,dy\leq C\frac{\mathfrak{f} M}{\rho^{2s}}M^{m-1}\leq C\frac{\mathfrak{f}M^{m}}{\rho^{2s}}.
\end{align*}
For the second term, we note that $\phi(k_i) \geq 0$ globally, and so we may estimate 
\begin{align*}
\int_{\bb{R}^n\setminus K_R} \frac{(\phi(u)-\phi(k_i))_{+}(y,t)}{|y|^{n+2s}}\,dy \leq \int_{\bb{R}^n\setminus K_R} \frac{\phi(u)_{+}(y,t)}{|y|^{n+2s}}\,dy \leq  \left(\frac{\rho}{R}\right)^{2s}\frac{\tl^m(u_{+},Q)}{\rho^{2s}}.
\end{align*}
Therefore enforcing
\begin{equation}\label{eq:2}
\left(\frac{\rho}{R}\right)^{\frac{2s}{m}}\tl(u_{+},Q) \leq \mathfrak{f}M,
\end{equation}
we get
\begin{align*}
\underset{x\in \mbox{supp}(\xi(\cdot,t)); t\in I_{i}}{\esup}\int_{\bb{R}^n\setminus K}\frac{(\phi(u)-\phi(k_i))_{+}(y,t)}{|x-y|^{n+2s}}\,dy \leq   C2^{i(n+2s)}M^{m-1}\frac{\mathfrak{f}M}{\rho^{2s}}.
\end{align*}
Hence, we find
\begin{align*}
\mathfrak{R}\leq CM^{m-1}\frac{2^{i(n+2s)}}{\rho^{2s}}(\mathfrak{f}M)^2|A_i|,
\end{align*}
where $A_i = \{u>k_i\}$. Next we have,
\begin{align*}
\mathfrak{Q} \leq C\frac{4^{is}}{\theta\delta \rho^{2s}}(\mathfrak{f}M)^2|A_i|,
\end{align*}
Similarly, using $M<u\leq 2M$ we get
\begin{align*}
\mathfrak{P} \leq CM^{m-1}\frac{4^{i}}{\rho^{2s}}(\mathfrak{f}M)^2|A_i|.
\end{align*}

Moreover, for the left-hand side of \cref{eq:cacc2}, using Lemma \ref{lem:est} we estimate
\begin{align*}
&M^{m-1}\int_{I_i}\int_{K_{i}}\int_{K_{i}}\min\{\xi^2(x,t),\xi^2(y,t)\}\frac{|(\widetilde{u}-k_i)_{+}(x,t)-(\widetilde{u}-k_i)_{+}(y,t)|^2}{|x-y|^{n+2s}}\,dx\,dy\,dt\\
&\leq C\int_{I_i}\int_{K_{i}}\int_{K_{i}}\min\{\xi^2(x,t),\xi^2(y,t)\}\frac{((\phi(u)-\phi(k_i))_{+}(x,t)-(\phi(u)-\phi(k_i))_{+}(y,t))((u-k_i)_+(x,t)-(u-k_i)_+(y,t))}{|x-y|^{n+2s}}\,dx\,dy\,dt 
\end{align*}
using that $|u| \leq 2M$, $k_i\in[\tfrac{M}{2},2M]$ and $m-1\geq 0$, where
\begin{align*}
\widetilde{u}:=\max\left\{u,M/4\right\}.
\end{align*}

Putting together the above estimates we arrive at
\begin{align}\label{eq:3}
\nonumber \esup_{t \in I_i}\int_{K_{i}}\xi^2(\widetilde{u}-k_i)^2_{+}\,dx
+M^{m-1}\int_{I_{i}}\int_{K_{i}}\int_{K_{i}} \min\{\xi^2(x,t),\xi^2(y,t)\}&\frac{|(\widetilde{u}-k_i)_{+}(x,t)-(\widetilde{u}-k_i)_{+}(y,t)|^2}{|x-y|^{n+2s}}\,dx\,dy\,dt \\
\leq C(\mathfrak{f}M)^2|A_i| \frac{2^{i(n+2)}}{\rho^{2s}}\left(2M^{m-1}+\frac{1}{\theta\delta}\right).  
\end{align}
		
Now, we observe
\begin{equation}\label{eq:31}
\begin{aligned}
\frac{(\mathfrak{f}M)^2}{4^{i+2}} |A_{i+1}| &\leq \iint_{Q_i} ((\widetilde{u}-k_i)\xi)_+^2\,dx\,dt\\
&\leq\left(\iint_{{Q}_i}((\widetilde{u}-{k}_i)\xi)_{+}^{2\frac{n+2s}{n}} \,dx\,dt\right)^{\frac{n}{n+2s}}|A_i|^{\frac{2s}{n+2s}}\\
&\leq \left(\int_{I_i}\int_{{K}_i}\int_{{K}_i} \frac{|((\widetilde{u}-{k}_i)\xi)_{+}(x,t)-((\widetilde{u}-{k}_i)\xi)_{+}(y,t)|^2}{|x-y|^{n+2s}} \,dx\,dy\,dt\right)^{\frac{n}{n+2s}}\\
&\quad\times\left(\esup_{t\in I_i}\int_{K_i} |((\widetilde{u}-{k}_i)\xi)^2_{+}(x,t)|\,dx\right)^{\frac{2s}{n+2s}}|A_i|^{\frac{2s}{n+2s}}.
\end{aligned}
\end{equation}
where we used the Sobolov embedding as in \eqref{eq:SP1}. 

Arguing as in earlier proofs, we get
\begin{align*}
Y_{i+1} \leq Cb^i\left(\frac{M^{1-m}}{\theta\delta}\right)^{\frac{-2s}{n+2s}}\left(2+\frac{M^{1-m}}{\theta\delta}\right)Y_i^{1+\frac{2s}{n+2s}}\quad\text{for}\quad Y_n = \frac{|A_n|}{|Q_n|}.
\end{align*}
Therefore 
\begin{align*}
Y_n \rightarrow 0,\quad\text{provided}\quad Y_0 \leq C\frac{M^{1-m}}{\theta\delta}\left(2+\frac{M^{1-m}}{\theta\delta}\right)^{-\frac{n+2s}{2s}}.
\end{align*}
Recalling that $\theta = M^{1-m}$, the conclusion follows. 
\end{proof}

\begin{lem}\label{prop:critical-mass'}
Let $0<m\leq 1$ and $\mathfrak{f} \in (0,1)$. Then we can find a $\mu' > 0$ independent of $\mathfrak{f}$, but depends on $\delta'$ and the \textbf{data} such that if 
\begin{align*}
|\{u \geq 2M-\mathfrak{f}M\} \cap Q'| \leq \mu' |Q'| \quad \mbox{ for } \quad Q' =   K_{\vartheta\rho\nu'^{-\frac{1}{2s}}} \times (\tau,\tau+\delta'\rho^{2s}]
\end{align*}
and 
\begin{align}\label{eq:tail.as2'}
\left(\frac{\vartheta\nu'^{-\frac{1}{2s}}\rho}{R}\right)^{\frac{2s}{m}}\tl(u_{+},Q) \leq \mathfrak{f}^{\frac{1}{m}}M
\end{align}
hold, then
\begin{align*}
u \leq 2M - \frac{1}{2}\mathfrak{f}M \quad \mbox{ a.e. in } K_{\frac{1}{2}\vartheta\rho\nu'^{-\frac{1}{2s}}} \times (\tau+\tfrac{3}{4}\delta'\rho^{2s},\tau+\delta'\rho^{2s}].
\end{align*}
\end{lem}

\begin{proof}
Without loss of generality we assume that $\tau+\delta\rho^{2s} = 0$.  For $i = 0, 1, \ldots$ we set
\begin{align*}
k_i = 2M-\frac{\mathfrak{f}M}{2}-\frac{\mathfrak{f}M}{2^{i+1}}, \quad \rho_i = \frac{\rho}{2}+\frac{\rho}{2^{i+1}}, \quad K_i = K_{\vartheta\nu^{-\frac{1}{2s}}\rho_i}, \quad Q_i = K_i \times I_i = K_i\times (-\delta'\rho_i^{2s},0].
\end{align*}
We note that on the set $\{u>k_i\}$ we have
\begin{align*}
M < u \leq 2M.
\end{align*}
We choose the cutoff function $\xi$ such that $0\leq \xi\leq 1$, $\xi\equiv 1$ on $Q_{i+1}$ with its support away from $\partial_P Q_i$ with 
\begin{align*}
|\nabla \xi|\leq C\frac{2^i}{\vartheta\nu^{-\frac{1}{2s}}\rho}\quad\text{and}\quad |\partial_t\xi|\leq C\frac{4^{is}}{\delta'\rho^{2s}}.
\end{align*}
Therefore, we may apply Lemma \ref{lem:ee.m<1} and arguing similarly to the proof of Lemma \ref{prop:critical-mass} using \cref{eq:tail.as2'} instead of \cref{eq:tail.as2} yields the conclusion. Again due to the space scaling there is an extra factor of $\nu'$ in the constant $\mu'$; more precisely, in this case $\mu' = C\nu'^{\frac{n}{2s}}(2+\delta^{-1})^{-\frac{n+2s}{2s}}\delta^{-1}$. 
\end{proof}
Now we are ready to prove the following.
\begin{prop}\label{prop:close-to-zero+}
Let $m\geq 1$ and $u$ be a local weak solution to \cref{eq:main-eq} in $\Omega_T$ in the sense of Definition \ref{def:weak}. Let $Q:=K_{R}\times I_{R} \subset \Omega_T$ and $M>0$ be such that 
\begin{align*}
2M \geq  \esup_{Q} |u|. 
\end{align*}
We set $\theta = M^{1-m}$ and suppose that $Q_{\rho}(\theta) \subset Q$. For the universal constant $\nu \in (0,1)$ as obtained in Lemma \ref{lem:deG-general}, if we have
\begin{align}\label{eq:measureQ}
|\{u \leq M\} \cap Q_{\rho}(\theta)| > \nu |Q_{\rho}(\theta)| 
\end{align}
then there exist a $\kappa \in (0,1)$ and a $\pi \in (0,1)$ which also depend only on the \textbf{data} such that 
\begin{align*}
u \leq 2M(1-\kappa) \quad  \mbox{ a.e. in } \quad Q_{\frac{\rho}{2}}(\nu\theta)
\end{align*}
provided
\begin{align}\label{eq:tail.pi}
\left(\frac{\rho}{R}\right)^{\frac{2s}{m}}\tl(u_+,Q) \leq \pi^{\frac{1}{m}} M.
\end{align}
\end{prop}
\begin{proof}
From Lemma \ref{lem:time-slice}, we have \cref{eq:time-slice} which gives the following measure density property for $2\rho$:
\begin{align}\label{eq:measureK}
\left|\left\{u(\cdot,\tau) \leq M \right\} \cap K_{2\rho}\right| > \frac{1}{2}\frac{\nu}{4^n}\left|K_{2\rho}\right|
\end{align}
for some time $\tau \in (-\theta(2\rho)^{2s},-\frac{\nu}{2}\theta(2\rho)^{2s}]$. With the choice of $\pi\leq 2^{-2s}$, the tail assumption in Lemma \ref{prop:time-exp} is satisfied and we get \cref{eq:time-exp} with $2\rho$ in place of $\rho$ and $4^{-n}\nu$ instead of $\nu$. This also fixes $\delta\in(0,1)$ and $\epsilon\in(0,\frac{1}{4})$ as in Lemma \ref{prop:time-exp}. Having fixed $\delta$, we now fix $\mu\in(0,1]$ as in Lemma \ref{prop:critical-mass}. After fixing $\mu$ we now choose a $j>1$ large enough so that
\begin{align*}
\frac{C}{\delta \nu} \frac{1}{2^j} \leq \mu,
\end{align*}
where $C$ is the universal constant in Lemma \ref{prop:shrink}. We further enforce that $\pi>0$ satisfy
\begin{align*}
\pi\leq \frac{\epsilon}{2^{j+2s}}.
\end{align*}
Then the tail assumption \cref{eq:tail.as} with $2\rho$ instead of $\rho$ in Lemma \ref{prop:shrink} is satisfied. Thus Lemma \ref{prop:shrink} implies that the measure hypothesis in Lemma \ref{prop:critical-mass} is satisfied for $\mathfrak{f} = \epsilon 2^{-j-1}$. We further enforce
\begin{align}\label{eq:pi}
\pi\leq\frac{\epsilon}{2^{j+1+2s}}
\end{align}
so that the tail assumption with $2\rho$ instead of $\rho$ in Lemma \ref{prop:critical-mass} is satisfied. Note that the previous two restrictions are now automatically satisfied. We get
\begin{equation}\label{eq:4}
u \leq 2M - \frac{\epsilon M}{2^{j+2}} \quad \mbox{ a.e. in } \quad K_{\rho} \times (\tau + \tfrac{3}{4}\delta\theta(2\rho)^{2s}, \tau + \delta\theta(2\rho)^{2s}].
\end{equation}

If 
\begin{align*}
\tau + \delta\theta(2\rho)^{2s} \geq 0,
\end{align*}
then the conclusion follows by setting $\kappa = \epsilon2^{-j-3}$. If we have not reached the top of the cylinder, i.e., $\tau + \delta\theta(2\rho)^{2s}<0$, then we can run the above argument again. However, this time we note that the conclusion \cref{eq:4} affords us a condition as in \cref{eq:time-slice} again  but this time with $\nu/2$ set to $1$.  We now repeat the forgoing argument to find a $\overline{\delta} \in (0,1)$ and $\overline{\epsilon} \in (0,1/4)$ which depends only on \textbf{data} this time (and is independent of $\nu$) and 
fix $\overline{\mu}\in(0,1]$ as in Lemma \ref{prop:critical-mass}. After fixing $\overline{\mu}$ we now choose a $\overline{j}>1$ large enough so that
\begin{align*}
\frac{C}{\overline{\delta}} \frac{1}{2^{\overline{j}}} \leq \overline{\mu},
\end{align*}
where $C$ is the universal constant in Lemma \ref{prop:shrink}. 


We further enforce
\begin{align}\label{eq:pi.tail}
\pi\leq\frac{\epsilon}{2^{j+\overline{j}+1+2s}}
\end{align}
and argue as before to get
\begin{equation*}
u \leq 2M - \frac{\overline{\epsilon}}{2^{\overline{j}+2}}\frac{\epsilon}{2^{j+1}}M \quad \mbox{ a.e. in } \quad K_{\rho} \times (\tau + \tfrac{3}{4}\delta\theta(2\rho)^{2s}, \tau + (\delta+\overline{\delta})\theta(2\rho)^{2s}].
\end{equation*}
We choose a $D \geq 1$ such that $D \overline{\delta} \geq 1$. Note that by construction of $\overline{\delta}$, the number $D$ is also universal. Therefore repeating the above argument $D$ times (but commonly using assumption \cref{eq:pi.tail} for tail) we arrive at
\begin{align*}
u \leq 2M - \left(\dfrac{\overline{\epsilon}}{2^{\overline{j}+1}}\right)^D\left(\dfrac{\epsilon}{2^{j+1}}\right)M \quad \mbox{ a.e. in } \quad K_{\rho} \times (\tau + \tfrac{3}{4}\delta\theta(2\rho)^{2s}, \tau + (\delta+D\overline{\delta})\theta(2\rho)^{2s}]
\end{align*}
By contruction,
\begin{align*}
\tau + (\delta+D\overline{\delta})\theta\rho^{2s} \geq 0
\end{align*}
and thus setting
\begin{align*}
\kappa = \frac{1}{2}\left(\dfrac{\overline{\epsilon}}{2^{\overline{j}+1}}\right)^D\left(\dfrac{\epsilon}{2^{j+1}}\right),
\end{align*}
the conclusion follows. 
\end{proof}

\begin{prop}\label{prop:close-to-zero+'}
Let $0<m\leq 1$ and $u$ be a local weak solution to \cref{eq:main-eq} in $\Omega_T$ in the sense of Definition \ref{def:weak}. Let $Q:=K_{R}\times I_{R} \subset \Omega_T$ and $M>0$ be such that 
\begin{align*}
2M \geq  \esup_{Q} |u|. 
\end{align*}
We set $\vartheta^{2s} = M^{m-1}$ and suppose that $\mathfrak{Q}_{\rho}(\vartheta) \subset Q$. For the universal constant $\nu' \in (0,1)$ as obtained in Lemma \ref{lem:deG-general'}, if we have
\begin{align}\label{eq:measureQ'}
|\{u \leq M\} \cap \mathfrak{Q}_{\rho}(\vartheta)| > \nu' |\mathfrak{Q}_{\rho}(\vartheta)| 
\end{align}
then there exist a $\kappa' \in (0,1)$ and a $\pi' \in (0,1)$ which also depend only on the \textbf{data} such that 
\begin{align*}
u \leq 2M(1-\kappa') \quad  \mbox{ a.e. in } \quad \mathfrak{Q}_{\frac{\rho}{2}}(\nu'^{-\frac{1}{2s}}\vartheta)
\end{align*}
provided
\begin{align}\label{eq:tail.pi'}
\left(\frac{\vartheta\nu'^{-\frac{1}{2s}}\rho}{R}\right)^{\frac{2s}{m}}\tl(u_+,Q) \leq \pi'^{\frac{1}{m}} M.
\end{align}
\end{prop}
\begin{proof}
By applying the similar argument in the proof of Proposition \ref{prop:close-to-zero+} using \cref{eq:measureQ'} and \cref{eq:tail.pi'} instead of \cref{eq:measureQ} and \cref{eq:tail.pi} respectively, we obtain the desired conclusion. The only difference is that in the choice of the first $\mu'$ we must have a $\nu'^2$ in the denominator instead of just $\nu'$ due to the scaling in space. 
\end{proof}

The following is the main goal of this section.
\begin{prop}\label{prop:close-to-0}
Let $m>0$ and $u$ be a local weak solution to \cref{eq:main-eq} in $\Omega_T$ in the sense of Definition \ref{def:weak}. Let $Q:=K_{R}\times I_{R} \subset \Omega_T$ and $M>0$ be such that 
\begin{align}\label{eq:sup.tail}
2M \geq  \esup_{Q} |u|+ \tl^{m}(u;Q)\text{ if }m\in(0,1] \quad\text{and}\quad 2M \geq  \esup_{Q} |u|+ \tl(u;Q) \text{ if }m\geq1.
\end{align}
Then we have the following.

\textbf{When $m\geq 1$: }Suppose that $(x_0,t_0)+Q_{\rho}(\theta)\subset Q$ and $\theta = M^{1-m}$. Fix the universal constant $\nu\in(0,1)$ as in Lemma \ref{lem:deG-general}. Then there exists a constant $\mathbf{c_0} \in (0,1)$ which depends only on \textbf{data} such that if 
\begin{align*}
|\{u \leq M\} \cap Q_{\mathbf{c_0}R}(\theta)| > \nu |Q_{\mathbf{c_0}R}(\theta)|\quad\text{and}\quad|\{u \geq - M\} \cap Q_{\mathbf{c_0}R}(\theta)| > \nu |Q_{\mathbf{c_0}R}(\theta)|
\end{align*}
hold, then there exists a $\kappa \in (0,1/2)$ which also depends only on the \textbf{data} such that 
\begin{align}\label{eq:concl}
|u| \leq 2M(1-\kappa) \quad  \mbox{ a.e. in } \quad Q_{\mathbf{c_0}R/2}(\nu\theta).
\end{align}

\textbf{When $0<m\leq 1$: }Suppose that $(x_0,t_0)+\mathfrak{Q}_{\rho}(\vartheta)\subset Q$ and $\vartheta = M^{\frac{m-1}{2s}}$. Fix the universal constant $\nu'\in(0,1)$ as in Lemma \ref{lem:deG-general'}. Then there exists a constant $\mathbf{c_0}' \in (0,1)$ which depends only on \textbf{data} such that if 
\begin{align*}
|\{u \leq M\} \cap \mathfrak{Q}_{\mathbf{c_0}'R}(\vartheta)| > \nu' |\mathfrak{Q}_{\mathbf{c_0}'R}(\vartheta)|\quad\text{and}\quad|\{u \geq - M\} \cap \mathfrak{Q}_{\mathbf{c_0}'R}(\vartheta)| > \nu' |\mathfrak{Q}_{\mathbf{c_0}'R}(\vartheta)|
\end{align*}
hold, then there exists a $\kappa' \in (0,1/2)$ which also depends only on the \textbf{data} such that 
\begin{align}\label{eq:concl'}
|u| \leq 2M(1-\kappa') \quad  \mbox{ a.e. in } \quad \mathfrak{Q}_{\mathbf{c_0}'R/2}(\nu'^{-\frac{1}{2s}}\vartheta).
\end{align}
\end{prop}

\begin{proof}
Note that
\begin{equation*}
(\mathbf{c_0}R/R)^{\frac{2s}m}\mathrm{Tail}(u_+;Q)\leq \mathbf{c_0}^{\frac{2s}{m}}M\quad\text{if }m\geq 1\quad\text{and}\quad(\mathbf{c_0}R \vartheta\nu'^{-\frac{1}{2s}}/R)^{\frac{2s}m}\mathrm{Tail}(u_+;Q)\leq (\mathbf{c_0\nu'^{-\frac{1}{2s}}})^{\frac{2s}{m}}M\quad\text{if }m\in(0,1].
\end{equation*}
where we used $\vartheta^{2s} = M^{m-1}$ for the second estimate. By taking $\mathbf{c_0}\leq \pi^{\frac1{2s}}$ and $\mathbf{c_0}'\leq \pi'^{\frac1{2s}}\nu'^{\frac{1}{2s}}$, where the constant $\pi$ is determined in Proposition \ref{prop:close-to-zero+} and $\pi'$ is determined in Proposition \ref{prop:close-to-zero+'}, we can apply Proposition \ref{prop:close-to-zero+} with $\rho$ replaced by $\mathbf{c_0}R$ and Proposition \ref{prop:close-to-zero+'} with $\rho$ replaced by $\mathbf{c_0}'R$, in order to prove \cref{eq:concl} and \cref{eq:concl'}.
\end{proof}

\section{H\"{o}lder Regularity}\label{sec:hol}

Now we are ready to prove the main theorem.
\begin{thm}\label{thm:hol2}
Let $u$ be a locally bounded, local weak solution to \cref{eq:main-eq} in $\Omega_T$. Then $u$ is locally H\"older continuous in $\Omega_T$. Moreover, there exist universal constants $C=C(\text{\textbf{data}})>0$ and $\alpha=\alpha(\text{\textbf{data}})\in (0,1)$ such that the following holds.

\textbf{When $m\geq 1$: } For any $0<r<R<R_0$ with $Q_R(\theta) \subset Q_{R_0} \subset \Omega_T$,
\begin{align}\label{eq:final-estimate-1}
\eosc_{(x_0,t_0) + Q_r(\theta)} u \leq CM\left(\frac{r}{R}\right)^{\alpha},
\end{align}
where $M =  \frac{1}{2}\esup_{(x_0,t_0) +Q_{R_0}} |u| + \tl(u;(x_0,t_0) +Q_{R_0})$ and $\theta:=M^{1-m}$.

\textbf{When $0<m\leq 1$: } For any $0<r<R<R_0$ with $\mathfrak{Q}_R(\vartheta) \subset Q_{R_0} \subset \Omega_T$,
\begin{align}\label{eq:final-estimate-2}
\eosc_{(x_0,t_0) + \mathfrak{Q}_r(\vartheta)} u \leq CM\left(\frac{r}{R}\right)^{\alpha},
\end{align}
where $M =  \frac{1}{2}\esup_{(x_0,t_0) +Q_{R_0}} |u| + \tl^m(u;(x_0,t_0) +Q_{R_0})$ and $\vartheta:={M^{\frac{m-1}{2s}}}$.
\end{thm}

Without loss of generality, we may assume that $(x_0,t_0) = (0,0)$. The proof of \cref{thm:hol2} is given with two cases: $m\geq1$ and $0<m\leq 1$.

\subsection{Porous Media}
Let $u$ be a solution to \cref{eq:main-eq} in $\Omega_T$. Let $Q=Q_R \subset \Omega_T$ and $M>0$ be such that 
\begin{align*}
M \geq  \frac{1}{2}\esup_{Q} |u| + \tl(u;Q). 
\end{align*}
We set $\theta_0=\theta = M^{1-m}$. In case of $m>1$, if $Q_{\rho}(\theta) \nsubseteq Q$, then $\theta\rho \geq R$ which implies that
\begin{align*}
M \leq \left(\frac{\rho}{R}\right)^{\frac{1}{m-1}}
\end{align*}
so that if we cannot find such a $\rho$ as $\rho\rightarrow 0+$, then $u = 0$ and is clearly H\"older continuous and so we are done. Therefore there is no loss of generality in assuming that 
\begin{align*}
Q_{\rho}(\theta) \subset Q.
\end{align*}
Note that the above assumption is always true when $m=1$. In particular, we have 
\begin{align*}
\esup_{Q_{\rho}(\theta)} |u| \leq 2M.
\end{align*}

\subsubsection{Close to $0$ ($m\geq 1$)}
Suppose that 
\begin{align*}
|\{u \leq M\} \cap Q_{\rho}(\theta)| > \nu |Q_{\rho}(\theta)| \quad\text{and}\quad |\{u \geq - M\} \cap Q_{\rho}(\theta)| > \nu |Q_{\rho}(\theta)|
\end{align*}
hold. We have
\begin{align*}
\tl(u,Q) \leq M
\end{align*}
by construction and so choosing $c>0$ small enough so that
\begin{align*}
c^{2s} \leq \pi,
\end{align*}
where $\pi \in (0,1)$ is as in Proposition \ref{prop:close-to-zero+}, and setting $\rho \leq cR$ we get that 
\begin{align*}
\left(\frac{\rho}{R}\right)^{\frac{2s}{m}}\tl(u;Q) \leq \pi^{\frac{1}{m}} M.
\end{align*}
Thus with $\kappa$ from Proposition \ref{prop:close-to-zero+},
\begin{align*}
\esup_{Q_{\rho/2}(\nu\theta)} |u| \leq 2M(1-\kappa).
\end{align*}
We now set
\begin{align}\label{eq:M0}
M_0=M,\quad M_1 = (1-\kappa)M \quad \mbox{and} \quad \theta_1 = M_1^{1-m}.
\end{align}
Choose a $0<\lambda\leq c$ such that 
\begin{align}\label{eq:lamda-inclusion}
\lambda\leq \frac{c}{2}\nu^{\frac{1}{2s}}(1-\kappa)^{\frac{m-1}{2s}}.
\end{align}
We set 
\begin{align}\label{eq:rho0}
\rho_0=\rho,\quad Q_0=Q_{\rho_0}(\theta_0),\quad \rho_1 = \lambda \rho_0 \quad  \quad Q_1 = Q_{\rho_1}(\theta_1) 
\end{align}
Then the choice of $\lambda$ enforces $Q_1 \subset Q_{\rho/2}(\nu\theta)$. Thus, 
\begin{align*}
\esup_{Q_1} |u| \leq 2M_1.
\end{align*}
Suppose that for $i=1,\ldots,l-1$ we have 
\begin{align*}
\rho_i=\lambda^i\rho_0,\quad M_i=(1-\kappa)^iM_0,\quad \theta_i=M_i^{1-m}\quad Q_i=Q_{\rho_i}(\theta_i) \quad \mbox{ and } \quad \esup_{Q_i} |u| \leq 2M_i,
\end{align*}
and we are close in zero in $Q_1,\ldots,Q_{l-1}$ i.e. for every $i = 1,\ldots, l-1$ we have
\begin{align*}
|\{u \leq M_i\} \cap Q_i| > \nu |Q_i| \quad\text{and}\quad |\{u \geq - M\} \cap Q_i| > \nu |Q_i|.
\end{align*}
We further restrict $\lambda$ by 
\begin{align}\label{eq:lambda-convergence}
    \lambda^{2s} \leq \frac{1}{2}(1-\kappa)^m.
\end{align}
for the forthcoming computation. To get a reduction in $Q_l$ we must now ensure that the far off effects are small. Indeed, by decomposing into successive annuli with respect to the radii $\rho_i$ and using the bound on $u$ in $Q_{i-1}$, we may compute 
\begin{align*}
    \tl^m(u,Q_i) &\leq C\rho_i^{2s}\sum_{j=1}^i\frac{M_{j-1}^m}{\rho_j^{2s}}\\
    &\leq CM_i^m\sum_{j=1}^i(1-\kappa)^{(i-j+1)(-m)}\lambda^{2s(i-j)}\\
    &\leq C_{\mathfrak{T}}M_i^m.
\end{align*}
where $C_{\mathfrak{T}}>1$ is a universal constant. Thus we may satisfy the tail condition by further restricting $\lambda$ by
\begin{align}\label{eq:lambda-tail}
\lambda^{2s} \leq \frac{\pi}{C_{\mathfrak{T}}}
\end{align}
Therefore, from \eqref{eq:lamda-inclusion},\eqref{eq:lambda-convergence} and \eqref{eq:lambda-tail}, choosing $\lambda > 0$ small enough so that
\begin{align}\label{eq:lambda-final}
\lambda^{2s} \leq \min \left\{\frac{(1-\kappa)^m}{2},\frac{\pi}{C_{\mathfrak{T}}}, \left(\frac{c}{2}\right)^{2s}\nu(1-\kappa)^{m-1}.\right\},
\end{align}
we can find a sequence of cylinders $Q_1,\ldots,Q_l$ such that if we are away from $0$ in $Q_1,\ldots,Q_{l-1}$ then for $j=1,\ldots,l$ we have
\begin{align*}
\esup_{Q_j} |u| \leq 2M \left(\frac{\rho_j}{R}\right)^{\alpha_0} \quad \mbox{ for } \quad \alpha_0 = \ln_{\lambda}(1-\kappa),
\end{align*}
where $\rho_j = \lambda^jR$, $M_j = (1-\kappa)^jM$, $\theta_j=M_j^{1-m}$ and $Q_j = Q_{\rho_j}(\theta_j)$. Therefore, if we always stay close to $0$, then we are done - we have shown that solutions that are small are H\"older. We now proceed to study the case where the solutions are away from $0$.

\subsubsection{Away from $0$ ($m\geq 1$)}
Suppose that $l$ is the first index such that $u$ is away from $0$ in $Q_l$. Then either 
\begin{align*}
|\{u \leq M_l\} \cap Q_{l}| \leq \nu |Q_{l}| \quad\text{or}\quad 	|\{u \geq M_l\} \cap Q_{l}| \leq \nu |Q_{l}| 
\end{align*}
holds. As above we compute
\begin{equation}\label{eq:5}
\tl^m(u;Q_l) \leq CM_l\left(1+\frac{\lambda^{2s}}{(1-\kappa)^m}+\cdots+\left(\frac{\lambda^{2s}}{(1-\kappa)^m}\right)^{l}\right) \leq 2CM_l.
\end{equation}
Therefore by Lemma \ref{lem:deG-general} and Remark \ref{rem:deG-general-tail} we have
\begin{align*}
|u| \geq \frac{1}{2}M_l   \quad \mbox{ a.e. in } \quad Q_{\tfrac{\rho_l}{2}}(\theta_l).
\end{align*}
By construction, we also have
\begin{align*}
\esup_{Q_l} |u| \leq 2M_l.
\end{align*}
We now perform a change of variables as follows:
\begin{align*}
X = \frac{x}{\rho_l} \quad \mbox{ and } \quad S = \frac{t}{\theta_l\rho_l^{2s}}
\end{align*}
and define
\begin{align*}
U(X,S) = \frac{1}{M_l}u(\rho_lX,\theta_l\rho_l^{2s}S) \quad \mbox{ where } \quad (X,S) \in Q_{\frac{1}{2}}.
\end{align*}
Then 
\begin{equation}\label{eq:6}
\frac{1}{2} \leq |U| \leq 2 \quad \mbox{ in } \quad Q_{1/2}.
\end{equation}
and we observe that $U$ satisfies 
\begin{equation}\label{eq:eq1}
\partial_t U+{L}\phi(U)=0\quad\text{in }Q_{1/2}.
\end{equation}
Finally setting $W = \phi(U)$ and using the bound \eqref{eq:6} we find a $\beta$ such that
\[
\beta(W) \sim W
\]
and
\[
\de_t \beta(W) - L'W = 0
\]
where 
\[
K'(X,Y) \sim |X-Y|^{-n-2s}.
\]
We note that we now have `elliptic' coefficients in time and a linear operator in space. It is well known how to handle `elliptic' and more general degenerate coefficients in time \cite{Par1,Par2} and so we may run linear arguments for the nonlocal part in the spirit of \cite{CCV11} to get oscillation decay:
\[
\eosc_{Q_r} W \leq Cr^{\alpha_1}
\]
for universal constants $C>0$ and $\alpha_1 \in (0,1)$. Note that the usual estimate would come with a $|W|_{\infty;Q_{1/2}}$ and a tail term on the right but we can estimate them from above by universal constants due to \eqref{eq:5} and \eqref{eq:6}. Using \eqref{eq:6} once again and going back to the original variables we get
\[
\eosc_{Q_r(\theta_l)} u \leq CM_l\left(\frac{r}{\rho_l}\right)^{\alpha_1}
\]
(with a different $C>0$ and universal). Therefore combining the close to $0$ and away from $0$ estimates and using that $m >1$ we arrive at
\[
\eosc_{Q_r(\theta)} u \leq CM\left(\frac{r}{R}\right)^{\alpha}
\]
for $\alpha = \min\{\alpha_0,\alpha_1\}$.

\subsection{Fast Diffusion}	
Again let $u$ be a solution to \cref{eq:main-eq} in $\Omega_T$. Let $Q=Q_R \subset \Omega_T$ and $M>0$ be such that 
\begin{align*}
M \geq  \frac{1}{2}\esup_{Q} |u| + \tl^m(u;Q). 
\end{align*}
We set $\vartheta_0^{2s}=\vartheta^{2s} = M^{m-1}$. In case of $m<1$, if $\mathfrak{Q}_{\rho}(\vartheta) \nsubseteq Q$ then $\vartheta\rho \geq R$ which implies that
\begin{align*}
M \leq \left(\frac{\rho}{R}\right)^{\frac{2s}{1-m}}
\end{align*}
so that if we cannot find such a $\rho$ as $\rho\rightarrow 0+$, then $u = 0$ and is clearly H\"older continuous and so we are done. Therefore there is no loss of generality in assuming that 
\begin{align*}
\mathfrak{Q}_{\rho}(\vartheta) \subset Q.
\end{align*}
In particular, we have 
\begin{align*}
\esup_{\mathfrak{Q}_{\rho}(\vartheta)} |u| \leq 2M.
\end{align*}
\subsubsection{Close to $0$ ($0<m\leq 1$)}
Suppose that 
\begin{align*}
|\{u \leq M\} \cap \mathfrak{Q}_{\rho}(\vartheta)| > \nu' |\mathfrak{Q}_{\rho}(\vartheta)| \quad\text{and}\quad |\{u \geq - M\} \cap \mathfrak{Q}_{\rho}(\vartheta)| > \nu' |\mathfrak{Q}_{\rho}(\vartheta)|
\end{align*}
hold. We have $\tl^m(u,Q) \leq M$ by construction, and so choosing $c'>0$ small enough so that $c'^{2s} \leq \pi'$, where $\pi' \in (0,1)$ is as in Proposition \ref{prop:close-to-zero+'} and setting $\rho \leq c'R$ we get that 
\begin{align*}
\left(\frac{\vartheta_0\nu'^{-\frac{1}{2s}}\rho}{R}\right)^{\frac{2s}{m}}\tl(u;Q) \leq \pi'^{\frac{1}{m}}M.
\end{align*}
Thus 
\begin{align*}
\esup_{\mathfrak{Q}_{\rho/2}(\nu'^{-\frac{1}{2s}}\vartheta)} |u| \leq 2M(1-\kappa'),
\end{align*}
where $\kappa'$ is from Proposition \ref{prop:close-to-zero+'}. We now set
\begin{align*}
M_0=M,\quad M_1 = (1-\kappa)M \quad \mbox{and} \quad \vartheta_1^{2s} = M_1^{m-1}
\end{align*}
and choose a $0<\lambda'\leq c'$ such that 
\begin{align*}
\lambda^{2s} \leq \frac{1}{\nu'}(1-\kappa')^{1-m}\left(\frac{c'}{2}\right)^{2s}. 
\end{align*}
We set 
\begin{align*}
\rho_0=\rho,\quad \mathfrak{Q}_0=\mathfrak{Q}_{\rho_0}(\vartheta_0),\quad \rho_1 = \lambda \rho_0 \quad \mbox{ and } \quad \mathfrak{Q}_1 = \mathfrak{Q}_{\rho_1}(\vartheta_1).
\end{align*}
Then the choice of $\lambda'$ enforces $\mathfrak{Q}_1 \subset \mathfrak{Q}_{\rho/2}(\nu'^{-\frac{1}{2s}}\vartheta)$. Thus, 
\begin{align*}
\esup_{\mathfrak{Q}_1} |u| \leq 2M_1.
\end{align*}
Suppose that for $i=1,\ldots,l-1$ we have 
\begin{align*}
\rho_i=\lambda'^i\rho_0,\quad M_i=(1-\kappa')^iM_0,\quad \vartheta^{2s}_i=M_i^{m-1}\quad \mathfrak{Q}_i=\mathfrak{Q}_{\rho_i}(\vartheta_i) \quad \mbox{ and } \quad \esup_{\mathfrak{Q}_i} |u| \leq 2M_i,
\end{align*}
and we are close in zero in $\mathfrak{Q}_1,\ldots,\mathfrak{Q}_{l-1}$ i.e. for every $i = 1,\ldots, l-1$ we have
\begin{align*}
|\{u \leq M_i\} \cap \mathfrak{Q}_i| > \nu |\mathfrak{Q}_i| \quad\text{and}\quad |\{u \geq - M\} \cap \mathfrak{Q}_i| > \nu |\mathfrak{Q}_i|.
\end{align*}
Then computations analogous to the ones for the case $m>1$ show that under the restriction 
\begin{align}\label{eq:lambda-final-2}
\lambda'^{2s} \leq \min \left\{\frac{(1-\kappa')^m}{2},\frac{\pi'}{C'_{\mathfrak{T}}}, \frac{1}{\nu'}\left(\frac{c'}{2}\right)^{2s}(1-\kappa')^{1-m}.\right\},
\end{align}
- where $c'>0$ and $C'_{\mathfrak{T}}>1$ are again universal constants - we can find a sequence of cylinders $\mathfrak{Q}_1,\ldots,\mathfrak{Q}_l$ such that if we are away from $0$ in $\mathfrak{Q}_1,\ldots,\mathfrak{Q}_{l-1}$ then for $j=1,\ldots,l$ we have
\begin{align*}
\esup_{\mathfrak{Q}_j} |u| \leq 2M \left(\frac{\rho_j}{R}\right)^{\alpha_0} \quad \mbox{ for } \quad \alpha_0 = \ln_{\lambda'}(1-\kappa'),
\end{align*}
where $\rho_j = \lambda'^jR$, $M_j = (1-\kappa')^jM$, $\theta_j=M_j^{1-m}$ and $Q_j = Q_{\rho_j}(\theta_j)$. Therefore, if we always stay close to $0$, then we are done - we have shown that solutions that are small are H\"older. We now proceed to study the case where the solutions are away from $0$.

\subsubsection{Away from $0$ ($0<m < 1$)}
Suppose that $l$ is the first index such that $u$ is away from $0$ in $Q_l$. Then either 
\begin{align*}
|\{u \leq M_l\} \cap \mathfrak{Q}_{l}| \leq \nu |\mathfrak{Q}_{l}| \quad\text{or}\quad 	|\{u \geq M_l\} \cap \mathfrak{Q}_{l}| \leq \nu |\mathfrak{Q}_{l}| 
\end{align*}
hold. As above we compute
\begin{equation}\label{eq:8}
\tl^m(u;\mathfrak{Q}_l) \leq CM_l\left(1+\frac{\lambda'^{2s}}{(1-\kappa')^m}+\cdots+\left(\frac{\lambda'^{2s}}{(1-\kappa')^m}\right)^{l}\right) \leq 2CM_l.
\end{equation}
Therefore by Lemma \ref{lem:deG-general'} and Remark \ref{rem:deG-general-tail} we have
\begin{align*}
|u| \geq \frac{1}{2}M_l   \quad \mbox{ a.e. in } \quad \mathfrak{Q}_{\tfrac{\rho_l}{2}}(\vartheta_l).
\end{align*}
By construction, we also have
\begin{align*}
\esup_{\mathfrak{Q}_l} |u| \leq 2M_l.
\end{align*}
We now perform a change of variables as follows. 
\begin{align*}
X = \frac{x}{\rho_l\vartheta_l} \quad \mbox{ and } \quad S = \frac{t}{\rho_l^{2s}}
\end{align*}
and define
\begin{align*}
U(X,S) = \frac{1}{M_l}u(\rho_l\vartheta_lX,\rho_l^{2s}S) \quad \mbox{ where } \quad (X,S) \in Q_{\frac{1}{2}}.
\end{align*}
Then 
\begin{equation}\label{eq:9}
\frac{1}{2} \leq |U| \leq 2 \quad \mbox{ in } \quad Q_{\frac{1}{2}}.
\end{equation}

Therefore, we observe that $U$ satisfies 
\begin{equation}\label{eq:eq11}
\partial_t U+{L}\phi(U)=0\quad\text{in }Q_{1/2}.
\end{equation}
Arguing as in the case $m>1$, we arrive at
\begin{align*}
\eosc_{\mathfrak{Q}_r(\vartheta_l)} u \leq CM_l\left(\frac{r}{\rho_l}\right)^{\alpha_1}
\end{align*}
for universal constants $C>0$ and $\alpha_1 \in (0,1)$. Therefore combining the close to $0$ and away from $0$ estimates and using that $m \in (0,1)$ we arrive at
\begin{align*}
\eosc_{\mathfrak{Q}_r(\vartheta)} u \leq CM\left(\frac{r}{R}\right)^{\alpha}
\end{align*}
for $\alpha = \min\{\alpha_0,\alpha_1\}$. This completes the proof of the case $0<m\leq 1$.

\begin{proof}[Proof of Corollary \ref{cor:Lio}]
Using \eqref{eq:final-estimate-1} and \eqref{eq:final-estimate-2}, for any $0<r<R$ we have 
\begin{align*}
\eosc_{(x_0,t_0)+Q_r(\theta)}u(x,t)\leq C\|u\|_{L^{\infty}(\mathbb{R}^n\times(-\infty,T))}\left(\frac{r}{\rho}\right)^{\alpha}
\end{align*}
for any $(x_0,t_0)+Q_r(\theta)\Subset \mathbb{R}^n\times(-\infty,T)$. Fixing $r$ and sending $R\rightarrow\infty$, we obtain the conclusion. 
\end{proof}


\bibliographystyle{alphaurl}
\bibliography{main}
\end{document}